\documentclass[12pt,reqno]{amsart}

\usepackage{mathptmx} 
\usepackage[shortlabels]{enumitem}

\usepackage[colorlinks,linkcolor={blue},
citecolor={blue},urlcolor={red},]{hyperref}

\usepackage{}
\usepackage{amsfonts}
\usepackage{amsmath, amsrefs, amsthm, amssymb}
\usepackage[active]{srcltx}		
\usepackage{pdfsync}					
\usepackage{graphicx}
\usepackage{bbm}
\usepackage{amsthm}
\usepackage{graphicx}
\usepackage{framed}
\usepackage{multicol,multienum}
\usepackage{graphicx} 
\usepackage{booktabs} 
\usepackage{mathrsfs}
\usepackage{tcolorbox}
\usepackage{bm}
\usepackage{subfigure}
\usepackage{epstopdf}

\newcommand{\dd}{\,{\rm d}}

\newcommand\R{{\mathbb{R}}}
\newcommand\T{{\mathbb{T}}}
\newcommand\N{{\mathbb{N}}}

\newcommand\Z{{\mathbb{Z}}}

\renewcommand\div{{\rm div}}

\newtheorem{theorem}{Theorem}[section]
\newtheorem{proposition}[theorem]{Proposition}
\newtheorem{lemma}[theorem]{Lemma}

\theoremstyle{definition}
\newtheorem{definition}[theorem]{Definition}

\newtheorem{assumption}[theorem]{Assumption}

\newtheorem{remark}[theorem]{Remark}

\numberwithin{equation}{section}

\renewcommand{\P}{\mathbb{P}}

\begin{document}

\title [Local and global solutions to SHKS]{Local pathwise solutions and regularization by noises for the stochastic hyperbolic Keller-Segel equation}\thanks{Corresponding author: Lei Zhang (lei\_zhang@hust.edu.cn)}

\author{Tengyu Li}
\address{School of Mathematics and Statistics, Hubei Key Laboratory of Engineering Modeling  and Scientific Computing, Huazhong University of Science and Technology,  Wuhan 430074, Hubei, P.R. China.}
\email{}

\author{Lei Zhang}
\address{School of Mathematics and Statistics, Hubei Key Laboratory of Engineering Modeling  and Scientific Computing, Huazhong University of Science and Technology,  Wuhan 430074, Hubei, P.R. China.}
\email{lei\_zhang@hust.edu.cn (L. Zhang)}


\date{\today}

\begin{abstract}
In this paper, we investigate the Cauchy problem associated with the stochastic hyperbolic Keller-Segel (SHKS) equation featuring multiplicative noises on the torus $\mathbb{T}^d$.
First, we establish the local existence and uniqueness of pathwise solutions to the SHKS equation within Sobolev spaces $H^s(\mathbb{T}^d)$ for $s>\frac{d}{2}+1$, under appropriate regularity conditions imposed on the nonlinear multiplicative noises. Subsequently, we explore two global results pertaining to noise-induced regularization: (1) The first result demonstrates that for polynomial-type nonlinear noises, when the noise intensity parameters meet specific threshold conditions, the SHKS equation possesses a unique pathwise solution for large initial data with probability one. This finding provides a partial answer to a question that has remained unresolved in the deterministic setting; (2) The second result reveals that, for small initial data, or equivalently when dealing with linear multiplicative noises with sufficiently large intensity (allowed to be negative), the SHKS equation admits a unique pathwise solution with high probability.
\end{abstract}

\maketitle
\section{Introduction}
\subsection{Formulation of problem} \label{110}
Chemotaxis refers to the oriented movement of cells (or an organism) in response to a chemical gradient. One of the most important equations used to describe this phenomena is the so-called Keller-Segel (KS) equations, which was originally derived in \cite{Keller-Segel}. Over the past decades, the deterministic KS equations have received much attention, 
see for example \cite{Hillen,bellomo2015}.
It is noteworthy that the macroscopic Keller-Segel equation is derived from the limiting behavior of microscopic models  \cite{stevens2000derivation}. In this framework, the density of cell is characterized by Fick's laws, where reaction dynamics are dependent on both cell density and chemoattractant concentration, while cell movement traces its microscopic origin to the irregular locomotion of individual organisms. Although fluctuations around the mean value are omitted during the derivation, realistic modeling necessitates the inclusion of random spatiotemporal forcing to account for the irreproducible characteristics of natural environments.
Another source of randomness arises from variations in environment or system parameters. 
As pointed out in \cite{patlak1953random,he2016noise} (see also \cite{hausenblas2022,huang2021}), noise in such biological systems can be classified into two categories: Internal noise stems from the inherent irregular movement of cells; External noise caused by random variations in environment.

As a result, incorporating random perturbations plays a crucial role in exploring the chemotaxis equation when it is subject to noisy environments. In this paper, we shall consider the Cauchy problem for the following stochastic hyperbolic Keller-Segel (SHKS) equation on the torus $\mathbb{T}^d$:
\begin{equation}\label{SHKS}
\left\{
\begin{aligned}
&\mathrm{d}u +\nabla(u(1-u)\nabla S)\mathrm{d}t= \sigma (t,u) \mathrm{d}W(t) ,   &&  x\in ~\mathbb{T}^d,t\ge 0,\\
& \Delta S=S-u,  &&  x\in ~\mathbb{T}^d,t\ge 0,\\
& u(0,x)=u_0(x), &&  x\in ~\mathbb{T}^d. \\
\end{aligned}
\right.
\end{equation}
Clearly, by using the second equation in \eqref{SHKS}, we have the relationship $S=(1-\Delta)^{-1}u$, which implies that   \eqref{SHKS} can be reformulated as
\begin{equation}\label{SHKS-trans}
\left\{
\begin{aligned}
&\mathrm{d}u +(1-2u)\nabla S\cdot \nabla u \mathrm{d}t+(u-u^2)\Delta S\mathrm{d}t=\sigma(t,u)\mathrm{d}W(t),   &&  x\in ~\mathbb{T}^d,t\ge 0,\\
& u(0,x)=u_0(x), & & x\in ~\mathbb{T}^d.\\
\end{aligned}
\right.
\end{equation}
Here in \eqref{SHKS} and \eqref{SHKS-trans}, the unknowns are the scalar functions $u=u(t, x)$, $S=S(t, x): \mathbb{R}^{+} \times \mathbb{R}^d \rightarrow \mathbb{R}^{+}$, which represent the density of cells and the chemical sensitivity, respectively.

The \emph{main purpose} of this paper is to study the existence and uniqueness of local and global strong solutions (in both the analytical and probabilistic senses) to the SHKS equation \eqref{SHKS} in Sobolev spaces. Let us present a concise overview of our results, with detailed statements available in Subsection 1.3: (1) If we assume that the random coefficient $\sigma(t,u)$ satisfies a general nonlinear condition (i.e., the locally Lipschitz continuous and suitable growth condition), then the SHKS equation \eqref{SHKS} admits a unique local-in-time pathwise solution as it is defined in Definition \ref{def-pathwise} below; (2) Based on the local theory established in Theorem \ref{result1}, we shall further show that the random noises with special structure have regularization effect on the pathwise solutions.

\subsection{Related works}\label{111}
In this subsection, we shall review several works in both of the deterministic and stochastic settings, which are closed related to the results obtained in this paper.

\textsc{Deterministic case.} In the deterministic setting, i.e., $\sigma \equiv 0$ in \eqref{SHKS}, the hyperbolic Keller-Segel (HKS) equations have received much attention during the past several years.  In fact, for one dimensional HKS equation with small diffusion $-\epsilon\Delta u$ under homogeneous Neumann boundary, Dolak and Schmeiser \cite{Dolak2005} established the global well-posedness with proper conditions on the initial data. Moreover, they also showed that, by taking the limit $\epsilon \rightarrow 0$, the solutions $u_{\epsilon}$  convergence to the entropy solution to the HKS equation in suitable topology. In higher dimensions, Burger et al. \cite{Burger2008} exhibited that under
suitable conditions, solutions of the parabolic version of HKS equation converge to solutions of the HKS as $\epsilon\rightarrow0$ for short times. There is a strong connection between this result and the work of \cite{Perthame2009}, where the authors performed the limit $\epsilon\rightarrow 0$ for weak solutions globally in time. For the Cauchy problem of parabolic version of HKS equation, Burger et al. \cite{Burger2006} not only proved the existence and uniqueness of weak solutions but also investigated the asymptotic behavior exhibited by the solutions when time becomes sufficiently large. A sub-threshold related to finite-time shock formation in solutions of the system HKS equation over $\mathbb{R}$ was presented by Lee and Liu \cite{Lee2015}.  In the framework of Besov spaces, Zhou et al. \cite{Zhou2021} first established local well-posedness for the Cauchy problem of HKS in $B^s_{p,r}$ ($1 \leq p,r \leq \infty$, $s > 1 + \frac{d}{p}$), and showed the data-to-solution map is not uniformly continuous. Subsequently, Zhang et al. \cite{Zhang2022} proved discontinuity of the solution map at $t=0$ (for $d=1$, $B^s_{2,\infty}$, $s > \frac{3}{2}$), established the Hadamard local well-posedness in $B^{1+d/p}_{p,1}$ ($1 \leq p < \infty$), and also derived two blow-up criteria for strong solutions via the Littlewood-Paley theory.

\textsc{Stochastic case.}  In recent years, taking the internal noises and the extended noises (as we mentioned in Subsection \ref{110}) into consideration,  many achievements have been made in the study of stochastic chemotaxis equations, see for example
\cite{hausenblas2022,chen2025,tang2024,huang2021,lukas2026,McKeanVlasov2021,Cattiaux2016,haskovec2011,chen2023,aryan2024,Rosenzweig2023}. Hausenblas et al. \cite{hausenblas2022} proved that the stochastic KS system (SKS) with homogeneous Neumann boundary conditions has a martingale solution in the finite time interval $[0, T]$,
and gave the regularity estimate of the solution. Chen et al. \cite{chen2025} studied systems with linear growth noise and nonlinear noise, and proved the existence and uniqueness of nonnegative global mild solutions. Tang and Wang \cite{tang2024}
investigated the existence, uniqueness, regularization effect and long time behavior of strong solutions for a nonlocal aggregation-diffusion equation with multiplicative noise on $\R^d$ ($d\ge 2$). Huang and Qiu \cite{huang2021} proved the existence of a unique nonnegative solution to stochastic Keller-Segel equations with common noise under a small $L^4$-norm condition on the initial data, and the divergence-free noise structure cancels the stochastic integral in the energy estimate, enabling control of the nonlinearity. Bol et al. \cite{lukas2026} overcame the challenges posed by degeneracy and aggregation on the whole space to obtain the well-posedness of a signal-dependent motility model and its mean-field limit from an interacting particle system. Cattiaux and P\'ed\`eches \cite{Cattiaux2016} proved the existence and uniqueness of a non-explosive strong Markov solution for the 2D stochastic Keller-Segel particle system in the subcritical regime $\chi < 4(1 - \frac{1}{N-1})$, while demonstrating finite-time blow-up for the supercritical case $\chi > 4$. Ha\v{s}kovec and Schmeiser \cite{haskovec2011} developed a stochastic particle method to approximate measure-valued solutions, including singularities for the 2D Keller-Segel system, and established its convergence using defect measures and BBGKY hierarchies. Chen, Gvozdik, and Li \cite{chen2023} rigorously derived the degenerate Keller-Segel system from a moderately interacting stochastic particle system via a nonlocal approximation and a vanishing diffusion limit $\sigma\to 0$. Aryan et al. \cite{aryan2024} proved that solutions to a broad class of active scalar flows with random diffusion on the torus converge exponentially fast in Gevrey norm to the uniform distribution. By introducing a random diffusion term, Rosenzweig and Staffilani \cite{Rosenzweig2023} successfully proved that the global solution of a class of generalized active scalar equations exists at high probability and has a monotonically decreasing Gevrey norm. Toma\v sevi\'c \cite{McKeanVlasov2021} proposed a probabilistic representation of the Keller-Segel model via a nonlinear SDE with singular interaction and established its global well-posedness in two dimensions for sufficiently small $\chi$. It is worth pointing out that the previous work mainly focused on the parabolic stochastic chemotaxis model, and there are still few studies on the hyperbolic-type chemotaxis equations, especially in the case of pathwise solutions.

\textsc{Global solutions.}  To the best of the authors'   knowledge, there are merely a few available studies addressing the issue of global well-posedness for the HKS equation in the framework of Sobolev spaces. In \cite{Meng2024}, Meng et al. first established a global well-posedness result for the HKS equation with small initial data near equilibrium state in $H^s(\mathbb{T}^d)$ for $s>1 + \frac{d}{2} $. Recently, in \cite{bi2025hyperbolic}, Bi and the second author of this paper extended the global theory in \cite{Meng2024} from Sobolev space to a broader class of Besov spaces, with the help of a distinct methodology that incorporates a weakly dissipative term $-\lambda u (\lambda>0)$ in HKS equation. It should be noted that the aforementioned results are proved by using analytical tools under a deterministic framework. One of the key \emph{novelties} of this paper lies in addressing the problem from the perspective of probability theory. Precisely, we shall show that solutions to the HKS equation with stochastic noise possessing nonlinear or linear structure are exactly global-in-time ones. Especially, we proved the existence and uniqueness of global pathwise solution for large initial data, which is left to be unsolved in the literatures. Let us mention that such a regularization by noise phenomena has already been widely explored for many other kinds of PDEs, see for example \cite{chen2025,glattholtz2014,tang2024,SMEP2,zhang2020local,gassiat2019regularization,flandoli2021high,hong2024regularization} and the reference cited therein.

\subsection{Notation and assumptions.}
Let us fix a separable Hilbert space $\mathfrak{U}$ and a stochastic basis
$
    \mathcal{S}=(\Omega,\mathcal{F},\{\mathcal{F}_t\}_{t \ge 0},\mathbb{P}),
$
where the filtration $\{\mathcal{F}_t\}_{t \ge 0}$ is right-continuous and complete with respect to $\P$.  Let $\{W(t)\}_{t\geq 0}$ be a $\mathcal {F}_t$-adapted $\mathfrak{U}$-valued cylindrical Wiener process defined by
$$
W(t )=\sum_{k\geq 1}W_k(t )e_k,
$$
where $\{W_k(t )\}_{k\geq1}$ is a family of mutually independent $\mathcal {F}_t$-adapted real-valued standard Wiener processes, and $\{e_k\}_{k\geq1}$ is a complete orthonormal basis in $\mathfrak{U}$. Introducing a larger auxiliary space
$
 \mathfrak{U}_0:=\{v=\sum_{k\geq 1}\alpha_j e_k;~\sum_{j\geq 1} \alpha_k^2/j^2  <\infty\} \supset \mathfrak{U}
$ endowed with the norm $\|v\|_{\mathfrak{U}_0}^2=\sum_{k\geq 1} \alpha_j^2/j^2$, it is well-known that the trajectories of $W(t)$ belongs to $C([0,T];\mathfrak{U}_0)$, $\mathbb{P}$-a.s.

Now consider another separable Hilbert space $X$, the space of Hilbert-Schmidt operators, denoted by $\mathcal{L}_2(\mathfrak{U};X)$, consists of all bounded linear operators $G:\mathfrak{U}\to X$ satisfying$\|G\|_{\mathcal{L}_2(\mathfrak{U};X)}^2=\sum_{k\ge 1}\|G e_k\|_{X}^2<\infty$, where $\{e_k\}$ is an orthonormal basis of $\mathfrak{U}$.

Let $L^2(\mathbb{T}^d)$ be the usual Hilbert space of square-integrable functions on the torus $\mathbb{T}^d$, and $H^s(\T^d)$ is the Sobolev space defined by
$$
H^s(\mathbb{T}^d)=\{f\in L^2(\mathbb{T}^d):\Vert f\Vert_{H^s(\mathbb{T}^d)}=\sum\limits_{k\in\mathbb{Z}^d}(1+k^2)^s\vert\widehat{f}(k)\vert^2<\infty\},
$$
which is endowed with the inner product
$
    (f,g)_{H^s}=\sum_{k\in\Z}(1+k^2)\widehat{f}(k)\overline{\widehat{g}(k)}=(\Lambda^s f,\Lambda^s g)_2,
$
where the Bessel potential $\Lambda^s=(1-\Delta)^{s/2}$. For two linear operators $A$ and $B$, their commutator is denoted by $[A,B]=AB-BA$.

For $1<p<\infty$ and $s\in\R$, the Banach space $L^p([0,T];H^s(\mathbb{T}^d))$ comprises all measurable functions $f:[0,T]\to H^s(\mathbb{T}^d)$ such that the norm $\| f\|_{L^p([0,T];H^s(\mathbb{T}^d))}=(\int_0^T\|f(t)\|_{H^s}^p\mathrm{d} t)^{\frac{1}{p}}<\infty$. Furthermore, for $\alpha\in(0,1)$, the fractional Sobolev space $W^{\alpha,p}([0,T];H^s(\T^d))$ is defined as the set of all functions $f\in L^p([0,T];H^s(\mathbb{T}^d))$ satisfying
$$
    [f]_{\alpha,p,T}^p=\int_0^T\int_0^T\frac{\|f(t)-f(\widetilde{t})_{H^s}^p}{|t-\widetilde{t}|^{1+2\alpha p}}\mathrm{d} t\mathrm{d} \widetilde{t}<\infty,
$$
and the space is equipped with the norm
$
    \|f\|_{W^{\alpha,p}([0,T];H^s(\T^d))}^p=\|f\|_{L^p([0,T];H^s(\T^d))}^p+[f]_{\alpha,p,T}^p.
$

\begin{assumption}\label{assumption}
For any $s>\frac{d}{2}+1$, we make the following assumptions in this paper

\textbf{(A1)} (Growth condition): For all $u\in H^s$, we suppose that there exists an increasing and locally bounded function $\beta:\mathbb{R}^+\mapsto [1,\infty)$ such that
\begin{equation}\label{A1}
    \Vert \sigma(t,u)\Vert_{\mathcal{L}_2(\mathfrak{U};H^s)}\le \beta(\Vert u\Vert_{W^{1,\infty}})(1+\Vert u\Vert_{H^s}),\quad\ \forall t\ge 0.
\end{equation}

\textbf{(A2)} (Locally Lipschitz continuity): For all $u$, $v\in H^s$, we suppose that there exists an increasing function $\gamma:\R^+\mapsto\R^+$, which is locally bounded and satisfies
\begin{equation}\label{A2}
    \Vert \sigma(t,u)-\sigma(t,v)\Vert_{\mathcal{L}_2(\mathfrak{U};H^s)}\le \gamma (\Vert u\Vert_{W^{1,\infty}}+\Vert v\Vert_{W^{1,\infty}})\Vert u-v\Vert_{H^s},\quad\ \forall t\ge 0.
\end{equation}
\end{assumption}

\begin{remark}
If the noise satisfies certain conditions, then Assumption \ref{assumption} involves two kinds of interesting special cases
$$
\sigma(t,u) =\sum_{i=1}^\infty c_i (1+\|u\|_{W^{1,\infty}})^{\delta} u , ~~ c\neq0,\quad \textrm{or} \quad \sigma(t,u) =\lambda u,~~ \lambda\neq 0.
$$
In what follows, we will further demonstrate that, by leveraging the local theory established in Theorem \ref{result1} below, the two aforementioned types of noises ensure that the corresponding solutions are indeed global in time (cf. Theorem \ref{result2} and Theorem \ref{result3}).
\end{remark}

\subsection{Main results}

To give the statement of our main results, let us first provide the definition of local pathwise solutions to the SHKS equation.

\begin{definition}\label{def-pathwise}
Let $s>\frac{d}{2}+1$ with $d\ge 1$ and $\mathcal{S}=(\Omega,\mathcal{F},\mathbb{P},\{\mathcal{F}_t\}_{t \ge 0},W)$ be a fixed stochastic basis as before. Assume that $u_0$ is an $H^s(\T^d)$-valued $\mathcal{F}_0$-measurable random variable with $\mathbb{E}\Vert u_0\Vert_{H^s}^2<\infty$.

(i) A \emph{local} pathwise solution of \eqref{SHKS-trans} is a pair $(u,\tau)$, where $\tau$ is a strictly positive stopping time, $u:\Omega\times[0,\infty)\to H^s(\mathbb{T}^d)$ is a $\mathcal{F}_t$-predictable process satisfying
$$
u(\cdot\land\tau)\in L^2(\Omega,C([0,\infty),H^s(\mathbb{T}^d))),
$$
and for all $t\ge 0$ and $\phi\in C^{\infty}(\mathbb{T}^d)$
\begin{equation}\label{pathwise-solutions}
\begin{aligned}
    &(u(t\land\tau),\phi)_2+\int_0^{t\land\tau}((1-2u)\nabla S\cdot\nabla u+(u-u^2)\Delta S,\phi)_2\mathrm{d}t\\
    &\quad \quad\quad\quad\quad\quad\quad\quad =(u_0,\phi)_2+\sum_{k\ge 1}\int_0^{t\land\tau}(\sigma(r,u)e_k,\phi)_2\mathrm{d}W_k(r),\quad\ \mathbb{P}\text{-a.s.}
\end{aligned}
\end{equation}

(ii) The local pathwise solutions are considered \textit{unique} or \emph{indistinguishable} if, for any two such solutions $(u_1, \tau_1)$ and $(u_2, \tau_2)$ satisfying $\mathbb{P}\{u_1(0) = u_2(0)\} = 1$, the following holds
\begin{equation}\label{unique}
    \mathbb{P}\{u_1(t,x)=u_2(t,x), \quad \forall (t,x)\in [0,\tau_1\land\tau_2]\times\mathbb{T}^d\}=1.
\end{equation}

(iii) A \emph{maximal} pathwise solution is a triple $(u,\{\tau_n\}_{n\ge 1},\xi)$, if each pair$(u,\tau_n)$ is a local pathwise solution, the sequence $\tau_n$ increases with $\lim_{n\to\infty} \tau_n=\xi$, and on the set $\{\xi < \infty\}$ the following holds $\mathbb{P}$-almost surely
\begin{equation}
    \sup_{t\in[0,\tau_n]}\Vert u(t)\Vert_{W^{1,\infty}}\ge n .
\end{equation}
If $\xi = \infty$ holds $\mathbb{P}$-almost surely, then the maximal pathwise solution is said to be \textit{global}.
\end{definition}

The first result in this paper is as follows.

\begin{theorem}[Local-in-time pathwise solution]\label{result1}
    Let $s>\frac{d}{2}+1$ with $d\ge 1$. Assume that $u_0$ is an $H^s(\T^d)$-valued $\mathcal{F}_0$-measurable random variable such that $\mathbb{E}\Vert u_0\Vert_{H^s}^2<\infty$, and the conditions \textbf{(A1)}-\textbf{(A2)} hold. Then the SHKS equation \eqref{SHKS-trans} has a unique local maximal pathwise solution $(u,\{\tau_n\}_{n\ge 1},\xi)$, in the sense of Definition \ref{def-pathwise}.
\end{theorem}
\begin{remark}
As noted previously, when compared to existing studies on the stochastic Keller-Segel equation, such as those in \cite{tang2024,chen2025,huang2021}, the primary challenge in proving Theorem \ref{result1} lies in the hyperbolic structure and nonlocal convection terms inherent to the equations under consideration. Specifically, the lack of a Laplacian operator results in solutions to the HKS system having lower spatial regularity. This reduced regularity invalidates the use of classical compactness principles (e.g., the Aubin-Lions lemma), which are typically employed to establish the convergence of approximate solutions.
Furthermore, as discussed in Subsection \ref{111}, relatively little is known about the global-in-time solvability of HKS equations within appropriate Sobolev spaces, particularly regarding the existence and uniqueness of global solutions for large initial data \cite{Meng2024}. Theorem \ref{result1} partially addresses this gap by providing an affirmative answer to this question, leveraging the powerful analytical tools associated with stochastic perturbations driven by Brownian-type noises.
The methodology of our proof draws key inspiration from three lines of work: \cite{glattholtz2014}, which focuses on the stochastic Euler equation; and \cite{zhang2020local}, which investigates the stochastic Euler-Poincar\'{e} equations. More precisely, our proof proceeds in the following steps:
First, we approximate the original equation using suitable cut-off functions and the classical Garlekin method. This approximation allows us to construct approximate solutions within sufficiently smooth Sobolev spaces.
Next, by establishing appropriate uniform a priori estimates, we prove the existence and uniqueness of local-in-time pathwise solutions via the stochastic compactness method.
Finally, leveraging an abstract stability lemma from \cite{glattholtz2014} (see Lemma \ref{abstract-cauchy-lemma} below), we demonstrate that these approximate solutions form a Cauchy sequence in the Sobolev space $H ^ s ( \mathbb{T}^d ) ( s > \frac{d}{2} + 1 )$. This property further implies that the limit of this sequence is indeed a local pathwise solution to the SHKS equation, as defined in Definition \ref{def-pathwise}.

Our next goal is to explore the regularization effect by suitable Brownian-type noises on the local pathwise solution constructed in Theorem \ref{result1}. The first one can be stated by the following theorem.
\end{remark}

\begin{theorem}[Global solution with large initial data]\label{result2}
   Fix $s>\frac{d}{2}+1$ with $d\ge 1$, and assume that $u_0$ is an $H^s(\T^d)$-valued $\mathcal{F}_0$-measurable random variable in $L^2(\Omega;H^s(\mathbb{T}^d))$.
    Suppose that the parameters $\delta$ and $\{c_i\}_{i\geq1}$ satisfy
  $$
     \delta>\frac{1}{2} ~\textrm{and}~ \sum_{i=1}^\infty c_i^2\neq 0,\quad 
   \textrm{or}~
       \delta=\frac{1}{2}~ \textrm{and} ~ \sum_{i=1}^\infty c_i^2>\kappa,$$
   where $\kappa$ is an universal constant in \eqref{5.3} below. Then the local pathwise solution $(u,\{\tau_n\}_{n\ge 1},\xi)$ to the system with a class of nonlinear multiplicative noise
   \begin{equation}\label{1.8}
       \left\{
       \begin{aligned}
           &\mathrm{d}u+((1-2u)\nabla S(u)\cdot \nabla u +(u-u^2)\Delta S(u))\mathrm{d}t\\
           &\quad\quad\quad\quad\quad\quad\quad \quad\quad =\sum_{i=1}^\infty c_i(1+\|u\|_{W^{1,\infty}})^{\delta}u~\mathrm{d}W_i(t), &&x\in \mathbb{T}^d,~t\ge 0,\\
           &u(0,x)=u_0(x),&&x\in\mathbb{T}^d,
       \end{aligned}
       \right.
   \end{equation}
   extends globally in time with probability one, i.e., $\mathbb{P}\{\xi=\infty\}=1$.
\end{theorem}

\begin{remark}\label{rem1.7}
It is important to note that, Theorem \ref{result2} solely focuses on scenarios where the exponent of nonlinear noise intensity satisfies
$\delta \geq 1/2$. In contrast, the investigation into the global problem with large initial data when
$\delta < 1/2$, a context that may require additional constraints on parameter $c_i$, has not yet been fully resolved and thus calls for further in-depth exploration.
\end{remark}

\begin{remark}
Theorem \ref{result2} demonstrates that the nonlinear and nonlocal (since the norm  $\|u\|_{W^{1,\infty}}$ relies on the full information of
$u$ over $\mathbb{T}^d$) random noise term $\sum_{i=1}^\infty c_i( 1 + \|u\|_{ W^{1, \infty}}) ^\delta u$,  exerts a regularization effect on the solution of the deterministic HKS system when its intensity is appropriately chosen. Furthermore, Theorem \ref{result2} appears to represent the first global well-posedness result for \emph{large initial data} in the context of the HKS equation, which can be regarded as an advancement over the global well-posedness result established in the recent work \cite{Meng2024}.
To elaborate, the analytical framework employed in \cite{Meng2024} is limited to addressing periodic problems with small initial data: specifically, it only applies to initial data in the Sobolev space $H^s(\mathbb{T}^d)$ ($s > 1 + \frac{d}{2}$) that lie in the vicinity of the equilibrium state. The present study overcomes this limitation by incorporating perspectives from probability theory through the introduction of suitable stochastic perturbations driven by Brownian-type noises. As discussed in Subsection \ref{110}, such stochastically perturbed HKS equations also offer a more natural and accurate mathematical description of chemotaxis phenomena occurring in real world environments.
\end{remark}

\begin{theorem}[Global solution with small initial data]\label{result3}
    Assume that $\{W(t)\}_{t\geq0}$ is a $\mathcal {F}_t$-adapted real-valued one dimensional  Wiener process.   Let $s>\frac{d}{2}+1$ with $d\ge 1$. Suppose that $u_0$ is an $H^s(\T^d)$-valued $\mathcal{F}_0$-measurable random variable in $L^2(\Omega;H^s(\mathbb{T}^d))$, and let $(u,\{\tau_n\}_{n\ge 1},\xi)$ be the corresponding maximal strong pathwise solution to the system
    \begin{equation}\label{linear-noise}
       \left\{
       \begin{aligned}
           &\mathrm{d}u+((1-2u)\nabla S(u)\cdot \nabla u+(u-u^2)\Delta S(u))\mathrm{d}t=\lambda u~\mathrm{d}W(t), &x\in \mathbb{T}^d,~t\ge 0,\\
           &u(0,x)=u_0(x),&x\in\mathbb{T}^d,~t=0.
       \end{aligned}
       \right.
   \end{equation}
   For any   $R>1$ and $\rho>2$, there exists a constant $\widetilde{C}=\widetilde{C}(s,d)>0$ such that if the initial data satisfies
   \begin{equation}\label{small_initial_data}
      \|u_0\|_{H^s}\le \frac{\lambda^2}{2 R\rho \widetilde{C}},\quad\ \P\text{-a.s.},
   \end{equation}
   then the maximal existence time $\xi$ satisfies $\P\{\xi=\infty\}\ge 1-\frac{1}{R\frac{\rho-1}{2\rho}}$, which infers that the pathwise solution to system \eqref{linear-noise} exists globally in time with high probability.
\end{theorem}

\begin{remark}
When compared with the deterministic work \cite{bi2025hyperbolic}, we arrive at the following   observation: The authors in \cite{bi2025hyperbolic} demonstrated that for the HKS equation incorporating weak dissipation $\lambda u$, a globally defined solution exists for small initial data-provided that the dissipation parameter
$\lambda>0$ is sufficiently large. In this context, Theorem \ref{result3} can be regarded as an extension of Theorem 1.2 and Theorem 1.3 from \cite{bi2025hyperbolic}. A key novelty of Theorem \ref{result3} lies in its relaxation of constraints on the dissipation term: unlike the deterministic framework in \cite{bi2025hyperbolic}, where the dissipation parameter was strictly limited to positive values, the stochastic dissipation term ``$\lambda \mathrm{d} W(t)$'' in \eqref{linear-noise} allows $\lambda$ to take negative values. This expansion of valid parameter ranges represents a distinct advancement in the analysis of the HKS equation under stochastic perturbations.
\end{remark}

\begin{remark}
By comparing Theorem \ref{result2} and Theorem \ref{result3}, we deduce that nonlinear multiplicative noise exerts a more pronounced regularization effect than linear multiplicative noise. This result demonstrates that, at least for the HKS equation, nonlinear multiplicative noise, when its intensity parameters meet suitable threshold conditions, possesses sufficient efficacy to suppress the singularities triggered by the nonlinear and nonlocal terms that are intrinsic to the HKS equations. On the other hand, the noise term shown in \eqref{linear-noise} can be interpreted as a degenerate version of the noise in \eqref{1.8}. This insight provides a partially affirmative answer to the question raised in Remark \ref{rem1.7}, particularly in the case where $\delta=0$ in \eqref{1.8}. Nevertheless, this outcome is accompanied by a trade-off: we can only prove the existence of global solutions for small initial data, and these solutions do not hold with fully general probability (i.e., their validity lacks the necessity of full probability coverage).
\end{remark}

The structure of this paper is arranged as follows: In Section \ref{section2}, we focus on constructing local pathwise solutions for initial data belonging to the Sobolev space
$H^s(\mathbb{T}^d)$ ($s>\frac{d}{2}+3$), and this construction is carried out using the Galerkin method combined with cut-off functions. Section \ref{section3} is dedicated to establishing the local well-posedness of the problem in $H^s(\mathbb{T}^d)$ ($s>\frac{d}{2}+1$), achieved by leveraging an abstract Cauchy Theorem. Subsequently, Section \ref{section4} proves the global existence results under a variety of noise conditions; the validity of these proofs depends on the more precise structural properties imposed on the random noises.

\section{Pathwise solution in sufficiently regular spaces.}\label{section2}

\subsection{Approximation scheme.}
We construct the approximation scheme as follows. For any $R>1$, let $\theta_R(x):[0,\infty)\to[0,1]$ be a $C^{\infty}$ cut-off function such that
\begin{equation}\label{theta_R}
    \theta_R(x)=\left\{
    \begin{aligned}
        &1, &&x\in [0,R],\\
        &0, &&x>2R.
    \end{aligned}
    \right.
\end{equation}
Then we consider the following problem by cutting the nonlinearities in \eqref{SHKS-trans}:
\begin{equation}\label{cutoff}
\left\{
\begin{aligned}
&\mathrm{d}u +G(u)\mathrm{d}t=F(u)\mathrm{d}W(t),   &&  x\in ~\mathbb{T}^d,t\ge 0,\\
& u(0,x)=u_0(x), &&  x\in ~\mathbb{T}^d,\\
\end{aligned}
\right.
\end{equation}
where
\begin{equation}
    \begin{aligned}
        G(u)&:=\theta_R\big(\|u\|_{W^{1,\infty}})\big((1-2u)\nabla S(u)\cdot\nabla u + (u-u^2)\Delta S(u)\big),\\
F(u)&:=\theta_R\big(\|u\|_{W^{1,\infty}}\big)\sigma(u).
    \end{aligned}
\end{equation}
In the Sobolev space $H^s(\T^d)$, we select a complete orthonormal basis $\{h_k\}_{k\ge 1}$. For all $n\ge1$, let $P_n$ denote the orthogonal projection operator, mapping $H^s$ onto span$\{h_1,h_2,...,h_n\}$, given by
\begin{equation}\label{def-p_n}
    P_n v=\sum_{k=1}^n(v,h_k)h_k,\quad\text{for all}\ v\in H^s.
\end{equation}
For any $s\ge0$ and any $0\le r\le s$ the projector satisfies the standard error estimate
\begin{equation}\label{eq:spectral-error}
\|v-P_n v\|_{H^r(\mathbb{T}^d)}\le Cn^{\,r-s}\|v\|_{H^s(\mathbb{T}^d)},\quad\  v\in H^s(\mathbb{T}^d),
\end{equation}
with a constant $C$ independent of $n$. And we have the following fundamental properties of $P_n$
\begin{align}
    P_n^2&=P_n,\\
    (P_n u,v)_2&=(u,P_n v)_2,\\
    \|P_n u\|_{H^s}&\le\|u\|_{H^s}.
\end{align}
Moreover, $P_n$ commutes with constant coefficient differential operators
\begin{equation}
    P_n(\partial^\alpha u)=\partial^\alpha(P_n u).
\end{equation}
In particular, it commutes with $\nabla$, $\Delta$, $\Lambda^s=(1-\Delta)^{\frac{s}{2}}$ and similar operators.

We consider the following Galerkin approximation scheme for \eqref{SHKS-trans}
\begin{equation}\label{eq:galerkin-system}
\left\{
\begin{aligned}
&\dd u_n(t) + G_n(u_n(t))\dd t = F_{n}(u_n(t))\dd W(t),&&  x\in ~\mathbb{T}^d,t\ge 0,\\
&u_n(0)=P_n u_0,&&  x\in ~\mathbb{T}^d,
\end{aligned}
\right.
\end{equation}
where the drift and diffusion are defined by projecting the original truncated operators
\begin{equation}
    \begin{aligned}
G_n(u_n)&:=\theta_R\big(\|u_n\|_{W^{1,\infty}}\big)P_n\Big( (1-2u_n)\nabla S(u_n)\cdot\nabla u_n + (u_n-u_n^2)\Delta S(u_n)\Big),\\
F_{n}(u_n)&:=\theta_R\big(\|u_n\|_{W^{1,\infty}}\big)P_n\big(\sigma(u_n)\big).
\end{aligned}
\end{equation}
For any $R > 1$ and $u \in C([0,T], H^s(\mathbb{T}^d))$, by \textbf{(A1)}-\textbf{(A2)} and the properties of $P_{n}$, it is not difficult to verify that there exists a constant $ C_{R} > 0$ such that
\begin{align}\label{N-bnd}
    \Vert G_{n}(u_{n})\Vert_{H^s}+\Vert F_{n}(u_{n})\Vert_{\mathcal{L}_2(\mathfrak{U};H^s)}\le C_{R}(1+\Vert u_{n} \Vert_{H^s}),
\end{align}
\begin{align}\label{N-lip}
    \|G_n(u_n)-G_n(v_n)\|_{H^s}+\|F_{n}(u_n)-F_{n}(v_n)\|_{\mathcal{L}_2(\mathfrak{U};H^s)}\le C_{R}\|u_n-v_n\|_{H^s}.
\end{align}
\begin{remark}
    \eqref{N-bnd} and \eqref{N-lip} implt that the system \eqref{eq:galerkin-system} may be considered as an SDE in finite dimensions, with globally Lipschitz drift and diffusion. According to the existence theory of SDE in Hilbert space (see \cite{gawarecki2011stochastic,prato1992}), we may infer that there exists a unique global in time solution $u_n$ to \eqref{eq:galerkin-system}, which is continuous in time, that is, $u_n\in C([0,\infty);H^s)$, $\P$-a.s.
\end{remark}
\subsection{A priori estimates}
In this subsection we carry out a priori estimates for approximate solutions. We need the folowing crucial lemmas.
\begin{lemma}[Algebra Property \cite{taylor1991}]\label{lem-algebra}
    Let $s>\frac{d}{2}+1$ and $f,g\in H^s(\mathbb{T}^d)$, then
    \begin{equation*}
        \|fg\|_{H^s}\le C_s\| f\|_{H^s}\|g\|_{H^s}.
    \end{equation*}
\end{lemma}
\begin{lemma}[Kato-Ponce \cite{kato1988}] \label{lem-kato}
    If $f\in H^s\cap W^{1,\infty}$, $g\in H^{s-1}\cap L^{\infty}$ for $s>0$, then there exist a $C_s>0$, such that
$$
        \Vert [\Lambda^s,f]g\Vert_{L^2}\le C_s(\Vert \Lambda^s f\Vert_{L^2}\Vert g\Vert_{L^{\infty}}+\Vert\nabla f\Vert_{L^{\infty}}\Vert \Lambda^{s-1}g\Vert_{L^2}).
$$
    If $f,g\in H^s\cap L^{\infty}$, then
$$
\Vert fg\Vert_{H^s}\le C_s(\Vert f\Vert_{H^s}\Vert g\Vert_{L^{\infty}}+\Vert f\Vert_{L^{\infty}}\Vert g\Vert_{H^s}).
$$
\end{lemma}
\begin{proposition}\label{prop3.2}
    Let $s>\frac{d}{2}+3$, $d\ge 1$, $p> 2$ and $\alpha \in(0,\frac{p-2}{2p})$, assume that the condition \textbf{(A1)} holds, and $u_0\in L^{2p}(\Omega;H^s(\T^d))$ is an $H^s$-valued $\mathcal{F}_0$ measurable random variable. Given any $R>1$, let $\{u_{n}\}_{n\in\N}$ be the sequence of solutions to the Galerkin approximation \eqref{eq:galerkin-system}. Then the sequence $\{ u_{n}\}_{n\in\N}$ is uniformly bounded in
 $$
 L^{2p}(\Omega;C([0,T];H^s))\cap L^{2p}(\Omega;W^{\alpha,2p}([0,T];H^{s-1})).
 $$
 Moreover, we have
 \begin{equation}\label{3.2-G}
    \sup_{n\in\N}\mathbb{E }\bigg\Vert \int_0^tG_n(u_{n})\mathrm{d}\tau \bigg\Vert_{W^{1,2p}([0,T];H^{s-1})}^{2p}\le C_{s,p,T,R,u_0},
\end{equation}
and
\begin{equation}\label{3.2-F}
    \sup_{n\in\N}\mathbb{E }\bigg\Vert \int_0^tF_{n}(u_{n})\mathrm{d}W(t) \bigg\Vert_{W^{\alpha,2p}([0,T];H^{s-1})}^{2p}\le C_{s,p,T,R,u_0,\alpha}.
\end{equation}

\end{proposition}
\begin{proof}\textsc{Step 1:}
Applying the operator $\Lambda^s$ to \eqref{eq:galerkin-system}, and then using the It\^o formula, we derive
\begin{equation}
    \begin{aligned}
        \Vert u_{n}(t)\Vert_{H^s}^2-\Vert u_{0}\Vert_{H^s}^2&=2 \sum_{k \geq 1} \int_0^t(\Lambda^{s}(F_{n}(u_{n})e_k), \Lambda^{s}u_{n})_{2} dW_k\\
        &\quad\ -2\int_0^t(\Lambda^s G_{n}(u_{n}),\Lambda^s u_{n})_2\mathrm{d}\tau+\int_0^t \Vert \Lambda^s F_{n}(u_{n})\Vert_{\mathcal{L}_2(\mathfrak{U};L^2)}^2\mathrm{d}\tau\\
        &=\sum_{k\ge 1}\int_0^t I_{1_k}\mathrm{d}W_k+\int_0^t I_2\mathrm{d}\tau+\int_0^t I_3\mathrm{d}\tau.
    \end{aligned}
\end{equation}
Taking a supremum over $t\in[0,T ]$ and using the Burkholder-Davis-Gundy inequality, we obtain
\begin{equation}\label{pp3.2-1}
    \mathbb{E}\sup_{t\in[0,T]}\Vert u_{n}(t)\Vert_{H^s}^2-\mathbb{E}\Vert u_{0}\Vert_{H^s}^2\le C\mathbb{E}\sum_{k\ge 1}\bigg(\int_0^T\vert I_{1_k}\vert^2 \mathrm{d}t \bigg)^{\frac{1}{2}}+\int_0^T\mathbb{E}\vert I_2\vert\mathrm{d}t +\int_0^T \mathbb{E}\vert I_3 \vert \mathrm{d}t.
\end{equation}
To bound the first term on the right-hand side of \eqref{pp3.2-1}, we use the assumption (\textbf{A1}) and the Cauchy-Schwarz inequality
\begin{equation}\label{pp3.2-estimate-J1}
\begin{aligned}
    \mathbb{E}\sum_{k\ge 1}\bigg(\int_0^T\vert I_{1_k}\vert^2 \mathrm{d}t \bigg)^{\frac{1}{2}}&= 2\mathbb{E}\sum_{k\ge1}\bigg(\int _0^T\vert\theta_R(\Vert u_{n}\Vert_{W^{1,\infty}})(\Lambda^s \sigma(t,u_{n})e_k,\Lambda^s u_{n})_2\vert^2\mathrm{d}t\bigg)^{\frac{1}{2}}\\
    &\le 2 \mathbb{E}\bigg(\int_0^T\theta_R^2(\Vert u_{n}\Vert_{W^{1,\infty}})\Vert \sigma(t,u)\Vert_{\mathcal{L}_2(\mathfrak{U};H^s)}^2\Vert u_{n}\Vert_{H^s}^2\mathrm{d}t\bigg)^{\frac{1}{2}}\\
    &\le \frac{1}{2}\mathbb{E}\sup_{t\in[0,T]}\Vert u_{n}\Vert_{H^s}^2+C\beta^2(2R)\int_0^T(1+\mathbb{E}\Vert u_{n}\Vert_{H^s}^2)\mathrm{d}t.
\end{aligned}
\end{equation}
 For the second term of \eqref{pp3.2-1},
utilizing the properties of $P_n$, we first express $I_2$ as
\begin{equation}\label{pp3.2-2}
     \begin{aligned}
         I_2&=-2\theta_R(\Vert u_{n}\Vert_{W^{1,\infty}})\int_{\mathbb{T}^d}\Lambda^s u_{n}\cdot\Lambda^s[P_{n}(P_{n}((1-2u_{n})\nabla S(u_{n}))\cdot\nabla P_{n}u_{n})]\mathrm{d}x\\
         &\quad\ -2\theta_R(\Vert u_{n}\Vert_{W^{1,\infty}})\int_{\mathbb{T}^d}\Lambda^s u_{n}\cdot\Lambda^s(u_{n}-u_{n}^2)\Delta S(u_{n})\mathrm{d}x\\
         &=I_{21}+I_{22},
     \end{aligned}
\end{equation}
where $S(u_{n})=(1-\Delta)^{-1}u_{n}$. Then we commute the operator $\Lambda^s$ with $P_{n}((1-2u_{n})\nabla S(u_{n}))$ and use the fact that $(P_{n}u,v)_2=(u,P_{n}v)_2$ for any $u,v\in L^2(\mathbb{T}^d)$
\begin{equation}
    \begin{aligned}
        |I_{21}|
        &=2\theta_R(\Vert u_{n}\Vert_{W^{1,\infty}})\bigg|\int_{\mathbb{T}^d}\Lambda^s P_{n}u_{n}\cdot \Lambda^s [P_{n}((1-2u_{n})\nabla S(u_{n}))\cdot\nabla P_{n}u_{n}]\mathrm{d}x\bigg|\\
        &\le 2\theta_R(\Vert u_{n}\Vert_{W^{1,\infty}})\int_{\mathbb{T}^d}\vert\Lambda^s P_{n}u_{n}\cdot[\Lambda^s,P_{n}((1-2u_{n})\nabla S(u_{n}))]\nabla P_{n}u_{n}\vert\mathrm{d}x\\
        &\quad \ +2\theta_R(\Vert u_{n}\Vert_{W^{1,\infty}})\int_{\mathbb{T}^d}\vert\Lambda^s P_{n}u_{n}\cdot P_{n}((1-2u_{n})\nabla S(u_{n}))\nabla \Lambda^sP_{n}u_{n}\vert\mathrm{d}x,\\
    \end{aligned}
\end{equation}
applying Lemma \ref{lem-kato} and the estimate $\Vert P_{n}u_{n}\Vert_{L^{\infty}}\le\Vert u_{n}\Vert_{L^{\infty}}$ for any $u_{n}\in L^{\infty}(\mathbb{T}^d)$, we obtain
\begin{equation}\label{pp3.2-3}
    \begin{aligned}
    |I_{21}|
        &\le C_{s,R}\Vert \Lambda^s P_{n}u_{n}\Vert_{L^2}\Vert[\Lambda^s,P_{n}((1-2u_{n})\nabla S(u_{n}))]\nabla P_{n}u_{n}\Vert_{L^2}\\
        &\quad \ +C_{s,R}\Vert \Lambda^s P_{n} u_{n}\Vert_{L^2}\Vert P_{n}((1-2u_{n}\nabla S(u_{n}))\Vert_{L^2}\Vert \nabla \Lambda^s P_{n} u_{n}\Vert_{L^2}\\
        &\le C_{s,R}(\Vert u_{n}\Vert_{L^{\infty}}+\Vert \nabla u_{n}\Vert_{L^{\infty}})\Vert u_{n}\Vert_{H^s}^2.
    \end{aligned}
\end{equation}
Again, we use Lemma \ref{lem-kato} to deduce that
\begin{equation}\label{pp3.2-4}
    \begin{aligned}
        |I_{22}|
        &\le C_{s,R}\Vert u_{n}\Vert_{H^s}(\Vert u_{n}-u_{n}^2\Vert_{H^s}\Vert\Delta(1-\Delta)^{-1}u_{n}\Vert_{L^{\infty}} +\Vert u_{N
        }-u_{n}^2\Vert_{L^{\infty}}\Vert\Delta(1-\Delta)^{-1}u_{n}\Vert_{H^s})\\
        &\le C_{s,R} \Vert u_{n}\Vert_{L^{\infty}}^2\Vert u_{n}\Vert_{H^s}^2.
    \end{aligned}
\end{equation}
Therefore, taking \eqref{pp3.2-3} and \eqref{pp3.2-4} into \eqref{pp3.2-2}, we get
\begin{equation}\label{pp3.2-estimate-J2}
    \mathbb{E}\int_0^t\vert I_2\vert\mathrm{d}\tau\le C_{s,R}\int_0^t\mathbb{E}\Vert u_{n}\Vert_{H^s}^2\mathrm{d}\tau.
\end{equation}
For $I_3$, it follows from the assumptions of $\sigma(t,\cdot)$ that
\begin{equation}\label{pp3.2-estimate-J3}
    \begin{aligned}
        \mathbb{E}\int_0^t\vert I_3\vert\mathrm{d}\tau&\le 2\mathbb{E}\int_0^t\theta_R(\Vert u_{n}\Vert_{W^{1,\infty}})\Vert u_{n}\beta^2(\Vert u_{n}\Vert_{W^{1,\infty}})(1+\Vert u_{n}\Vert_{H^s}^2)\mathrm{d}\tau\\
        &\le 2\beta^2(2R)\int_0^t(1+\Vert u_{n}\Vert_{H^s}^2)\mathrm{d}\tau.
    \end{aligned}
\end{equation}
Combining \eqref{pp3.2-estimate-J1}, \eqref{pp3.2-estimate-J2} and \eqref{pp3.2-estimate-J3}, we can estimate \eqref{pp3.2-1} by
\begin{equation}
    \mathbb{E}\sup_{t\in[0,T]}\Vert u_{n}\Vert_{H^s}^2\le 2\mathbb{E}\Vert u_0\Vert_{H^s}^2+C_{s,R}\int_0^T\bigg(1+\mathbb{E}\sup_{\tau\in[0,t]}\Vert u_{n}(\tau)\Vert_{H^s}^2 \bigg)\mathrm{d}t,
\end{equation}
which leads to $\mathbb{E}\sup_{t\in[0,T]}\Vert u_{n}(t)\Vert_{H^s}^2\le C_{s,R,u_0,T}.$

\textsc{Step 2:} We deal with the high-order moment estimates for $\Vert u_{n}\Vert_{H^s}^{2p}$, $p\ge 2$. For this purpose, we fix any $p$ and compute $\mathrm{d}\Vert u_{n}\Vert_{H^s}^{2p}=\mathrm{d}(\Vert u_{n}\Vert_{H^s}^2)^p$ by using the it\^o formula
\begin{equation}
    \begin{aligned}
        \|u_{n}\|_{H^s}^{2p}-\|u_0\|_{H^s}^{2p} &=p\|u_{n}\|_{H^s}^{2p-2}\bigg(\sum_{k\ge 1}\int_0^t I_{1_k}\mathrm{d}W_k+\int_0^t I_2\mathrm{d}\tau+\int_0^t I_3\mathrm{d}\tau\bigg)\\
        &\quad\ \quad\ +\frac{p(p-1)}{2}\sum_{k\ge 1}\int_0^t \|u_{n}\|_{H^s}^{4(p-1)}I_{1_k}^2\mathrm{d} \tau.
    \end{aligned}
\end{equation}
Following a similar procedure, we apply the Burkholder-Davis-Gundy inequality upon taking the supremum over $t \in [0, T]$, which yields
    \begin{align}\label{pp3.2-Lp1}
        &\mathbb{E}\sup_{t\in[0,T]}\Vert u_{n}(t)\Vert_{H^s}^{2p}-\mathbb{E}\Vert u_0\Vert_{H^s}^{2p}\nonumber\\
        &\quad\ \le C_{p}\bigg(\mathbb{E}\sum_{k\ge 1}\bigg(\int_0^T \Vert u_{n}\Vert_{H^s}^{4(p-1)}\vert I_{1_k}\vert^2\mathrm{d}t \bigg)^{\frac{1}{2}}+\int_0^T\mathbb{E}(\Vert u_{n}\Vert_{H^s}^{2(p-1)}\vert I_2\vert)\mathrm{d}t\\
        &\quad\ \quad\ +\int_0^T\mathbb{E}(\Vert u_{n}\Vert_{H^s}^{2(p-1)}\vert I_3\vert)\mathrm{d}t+\int_0^T\mathbb{E}\sum_{k\ge 1}(\Vert u_{n}\Vert_{H^s}^{2(p-2)}\vert I_{1_k}\vert^2)\mathrm{d}t\bigg).\nonumber
    \end{align}
A similar estimate to \eqref{pp3.2-estimate-J1} gives
     \begin{align}
        &\mathbb{E}\sum_{k\ge 1}\bigg(\int_0^T \Vert u_{n}\Vert_{H^s}^{4(p-1)}\vert I_{1_k}\vert^2\mathrm{d}t \bigg)^{\frac{1}{2}}\nonumber\\
        &\quad\ \le C\mathbb{E}\bigg( \int_0^T\Vert u_{n}\Vert_{H^s}^{2p}\theta_R^2(\Vert u_{n}\Vert_{W^{1,\infty}})\beta^2(\Vert u_{n}\Vert_{W^{1,\infty}})(1+\Vert u_{n}\Vert_{H^s}^{2p})\mathrm{d}t\bigg)^{\frac{1}{2}}\nonumber\\
          &\quad \ \le C\mathbb{E}\bigg(\sup_{t\in[0,T]}\Vert u_{n}\Vert_{H^s}^{2p}\int_0^T\theta_R^2(\Vert u_{n}\Vert_{W^{1,\infty}})\beta^2(\Vert u_{n}\Vert_{W^{1,\infty}})(1+\Vert u_{n}\Vert_{H^s}^{2p})\mathrm{d}t\bigg)^{\frac{1}{2}}\\
        &\quad\ \le\frac{1}{2}+C\beta^2(2R)\int_0^T(1+\mathbb{E}\Vert u_{n}\Vert_{H^s}^{2p})\mathrm{d}t.\nonumber
    \end{align}
Using estimates analogous to those in \eqref{pp3.2-estimate-J2} and \eqref{pp3.2-estimate-J3}, we have
\begin{equation}
    \int_0^t\mathbb{E}\Vert u_{n}\Vert_{H^s}^{2(p-1)}\vert I_2\vert\mathrm{d}\tau\le C_{p,R}\int_0^t(1+\mathbb{E}\Vert u_{n}\Vert_{H^s}^{2p})\mathrm{d}\tau,
\end{equation}
\begin{equation}
    \int_0^t\mathbb{E}\Vert u_{n}\Vert_{H^s}^{2(p-1)}\vert I_3\vert\mathrm{d}\tau\le C_{p,R}\int_0^t(1+\mathbb{E}\Vert u_{n}\Vert_{H^s}^{2p})\mathrm{d}\tau,
\end{equation}
\begin{equation}\label{pp3.2-Lp2}
    \sum_{k\ge 1}\int_0^t\mathbb{E}\Vert u_{n}\Vert_{H^s}^{2(p-2)}\vert I_{1_k}\vert^2\mathrm{d}\tau\le C_{p,R}\int_0^t(1+\mathbb{E}\Vert u_{n}\Vert_{H^s}^{2p})\mathrm{d}\tau,
\end{equation}
we deduce from the estimates \eqref{pp3.2-Lp1}-\eqref{pp3.2-Lp2} that
\begin{equation}
    \mathbb{E}\sup_{t\in[0,T]}\Vert u_{n}\Vert_{H^s}^{2p}\le 2\mathbb{E}\Vert u_0\Vert_{H^s}^{2p}+C_{s,p,R}\int_0^T\bigg(1+\mathbb{E}\sup_{\tau\in[0,t]}\Vert u_{n}(\tau)\Vert_{H^s}^{2p} \bigg)\mathrm{d}t.
\end{equation}
By applying the Gronwall inequality we get
\begin{equation}\label{L2p-estimate}
    \mathbb{E}\sup_{t\in[0,T]}\Vert u_{n}\Vert_{H^s}^{2p}\le C_{s,p,R,T,u_0,p}.
\end{equation}
Moreover, the classical theory of SDEs with Lipschitz coefficients guarantees that $u_{n}$ is continuous in time, combining \eqref{L2p-estimate} and the Kolmogorov's continuity theorem, we find that $u_{n}\in L^{2p}(\Omega;C^{\alpha}([0,T];H^{s-1}))$ is bounded uniformly in $n$ for $\alpha \in\bigg(0,\frac{p-2}{2p}\bigg)$.

\textsc{Step 3:} In stochastic case, however, differentiability in time cannot generally be expected. Instead, we have to consider estimates involving fractional time derivatives of order less than $\frac{1}{2}$. Notice the fact that $W^{1,p}\hookrightarrow W^{\alpha,p}$ for any $\alpha\in(0,1)$, and the inequality $(a+b)^p\le 2^{p-1}(a^p+b^p)$, we obtain
    \begin{align}\label{3parts}
        \|u_{n}\|_{W^{\alpha,2p}([0,T];H^{s-1})}^{2p}&\le 3^{2p-1}\bigg(\|u_0\|_{H^s}^{2p}+\bigg\|\int_0^t G_{n}(u_{n}(r))\mathrm{d}r \bigg\|_{W^{1,2p}([0,T];H^{s-1})}^{2p}\\
        &\quad\ +\bigg\|\int_0^t F_{n}(u_{n}(r))\mathrm{d}W(r)\bigg\|_{W^{\alpha,2p}([0,T];H^{s-1})}^{2p}\bigg).\nonumber
    \end{align}
Applying the Minkowski inequality, Lemma \ref{lem-kato} and \eqref{L2p-estimate}, we deduce that
    \begin{align*}
        &\mathbb{E}\bigg\|\int_0^t G_{n}(u_{n}(r))\mathrm{d}r \bigg\|_{W^{1,2p}([0,T];H^{s-1})}^{2p}\\
        &\quad\ \le C_{p,T,R}\mathbb{E}\int_0^T(\|P_{n}((1-2u_{n})\nabla S(u_{n})\cdot \nabla u_{n})\|_{H^{s-1}}^{2p} +\|(u_{n}-u_{n}^2)\Delta S(u_{n})\|_{H^{s-1}}^{2p})\mathrm{d}t\\
        &\quad\ \le C_{s,p,T,R}\mathbb{E}\int_0^T(\|u_{n}\|_{L^{\infty}}+\|\nabla u_{n}\|_{L^{\infty}})^{4p}\|u_{n}\|_{H^{s-1}}^{2p}\mathrm{d}t\\
        &\quad\ \le C_{s,p,T,R}\int_0^T\mathbb{E}\|u_{n}\|_{H^{s-1}}^{2p}\mathrm{d}t\le C_{s,p,T,R,u_0}.\\
    \end{align*}
For the term involving $F_{n}(u_{n})$, we notice that for any $k>0$, there exists a interval $[t_k,\widetilde{t}_k]$ ($0\le t_k< \widetilde{t}_k\le T$), such that
$
    \sup_{0\le t<\widetilde{t}\le T}\frac{\|\int_t^{\widetilde{t}} F_{n}(u_{n})\mathrm{d}W(r)\|_{H^{s-1}}^{2p}}{|t-\widetilde{t}|^{1+2\alpha p}}\le \frac{\|\int_{t_k}^{\widetilde{t}_k} F_{n}(u_{n})\mathrm{d}W(r)\|_{H^{s-1}}^{2p}}{|t_k-\widetilde{t}_k|^{1+2\alpha p}}+k.
$
After taking the expectation on both sides of the above formula, we use \eqref{L2p-estimate} and the Burkholder-Davis-Gundy inequality again, we obtain
\begin{equation}
    \begin{aligned}
        &\mathbb{E}\sup_{0\le t<\widetilde{t}\le T}\frac{\|\int_t^{\widetilde{t}} F_{n}(u_{n}\mathrm{d}W(r)\|_{H^{s-1}}^{2p}}{|t-\widetilde{t}|^{1+2\alpha p}}\\
        &\quad\ \le \mathbb{E}\frac{\|\int_{t_k}^{\widetilde{t}_k} F_{n}(u_{n}\mathrm{d}W(r)\|_{H^{s-1}}^{2p}}{|t_k-\widetilde{t}_k|^{1+2\alpha p}}+k\\
        &\quad\ \le C_{p}|t_k-\widetilde{t}_k|^{-1-2\alpha p}\mathbb{E}\bigg(\int_{t_k}^{\widetilde{t}_k} \theta_R(\|u_{n}\|_{W^{1,\infty}})\beta(\|u_{n}\|_{W^{1,\infty}}(1+\|u_{n}\|_{H^{s-1}}^2))\mathrm{d}r\bigg)^{p}+k\\
        &\quad\ \le C_{p,R}|t_k-\widetilde{t}_k|^{-1-2\alpha p}(1+\mathbb{E}\sup_{r\in[0,T]}\|u_{n}\|_{H^{s-1}}^{2p})+k\\
        &\quad\ \le C_{s,p,T,R,u_0,\alpha}+k.
    \end{aligned}
\end{equation}
Since $k>0$ can be taken arbitrarily, and $p>1+2\alpha p$, thus we have
\begin{equation}
    \begin{aligned}
    \mathbb{E}\int_0^T\int_0^T\frac{\|\int_t^{\widetilde{t}}F_{n}(u_{n})\mathrm{d}W(r)\|_{H^{s-1}}^{2p}}{|t-\widetilde{t}|^{1+2\alpha p}}\mathrm{d} t\mathrm{d} \widetilde{t}\le C_{s,p,T,R,u_0},
    \end{aligned}
\end{equation}
combining with the estimate \eqref{L2p-estimate}, we deduce that
\begin{equation}
    \mathbb{E}\bigg\|\int_0^t F_{n}(u_{n})\mathrm{d}W(r)\bigg\|_{W^{\alpha,2p}([0,T];H^{s-1})}^{2p}\le C_{s,p,T,R,u_0,\alpha}.
\end{equation}
According to \eqref{3parts} and the fact that $u_0\in L^{2p}(\Omega;H^s(\mathbb{T}^d))$, we   get that
$
    \mathbb{E}\|u_{n}\|_{W^{\alpha,2p}([0,T];H^{s-1})}^{2p}\le C_{s,p,T,R,u_0,\alpha},
$
which completes the proof of Proposition \ref{prop3.2}.
\end{proof}

\subsection{Tightness of measures}
For a given initial distribution $\mu_0$ on $H^s$, we fix a stochastic basis $\mathcal{S}=(\Omega,\mathcal{F},\mathbb{P},\{\mathcal{F}_t\}_{t \ge 0},W)$ upon which is defined an $\mathcal{F}_0$-measurable random element $u_0$ with
distribution $\mu_0$. In order to define a sequence of measures associated with $\{(u_{n},W)\}_{N >0}$, we consider the phase space
$$
\mathcal{X}=\mathcal{X}_S\times \mathcal{X}_{W},\quad \text{where}\ \mathcal{X}_S=C([0,T];H^{s-1}(\mathbb{T}^d)),\ \mathcal{X}_W=C([0,T];\mathfrak {U}).
$$
We may consider the first component, $\mathcal{X}_S\supset C([0,T];H^{s}(\mathbb{T}^d))$, and the second component, $\mathcal{X}_W$, as being the space on which the driving Brownian motions are defined. On $\mathcal{X}$ we define the probability measures
\begin{equation}\label{def-mu}
    \mu_{n}=\mu_{n}^S\times\mu_W, \quad \ \text{where}\ \mu_{n}^S=\mathbb{P}(\mu_{n}\in\cdot), \ \mu_W(\cdot)=\mathbb{P}(W\in\cdot).
\end{equation}

Let $\mathcal{P}r(\mathcal{X})$ be the collection of Borel probability measures on $\mathcal{X}$. Recalling that a set $\mathcal{O}\subset \mathcal{P}r(\mathcal{X})$ is \emph{tight} if, for every $\theta>0$, there exists a compact set $K_{\gamma}\subset \mathcal{X}$ such that, $\mu(K_{\gamma})\ge 1-\gamma$ for all $\mu \in \mathcal{O}$.
\begin{lemma}[Compactness Theorem \cite{flandoli1995}] \label{flandoli1995}
    Suppose that $B_0$, $B$ are Banach spaces, and the embedding from $B_0$ into $B$ is compact. Let $p\in(1,\infty)$ and $\alpha\in(0,1]$ be such that $\alpha p > 1$.
Then $W^{\alpha,p}([0, T] ; B_0) \subset C([0, T ]; B)$ and the embedding is compact.
\end{lemma}
Based on the above compactness criteria, we have the following tightness result.

\begin{proposition}\label{prop3.3}
    Let $s >\frac{d}{2}+3$, $p>2$, $R\ge 1$ and $\mu_0\in \mathcal{P}r(H^{s})$ with $\int_{H^{s}}\Vert u\Vert_{H^{s}}^{p}\mathrm{d}\mu_0(u)<\infty$. Suppose that \textbf{(A1)} holds, and $u_0$ is a $\mathcal{F}_0$-measurable random variable defined on the stochastic basis $\mathcal{S}$ with distribution $\mu_0$. Then the sequence $\{\mu_{n}\}_{n\in\N}$ defined by \eqref{def-mu} is tight in $\mathcal{P}r(\mathcal{X} )$.
\end{proposition}
\begin{proof} Fix any $\alpha\in(0,\frac{p-2}{2p})$ with $2\alpha p>1$, it follows from Lemma \ref{flandoli1995} that $W^{1,2p}([0,T];H^{s-1})$ and $W^{\alpha,2p}([0,T];H^{s-1})$ are compactly embedded in $\mathcal{X}_S$, which implies that the ball with radius $r>0$,
    $
    B_{n}(r)=\{u_{n}:\Vert u_{n}\Vert_{W^{1,2p}([0,T];H^{s-1})}\le r\}\cap\Vert u_{n}\Vert_{W^{\alpha,2p}([0,T];H^{s-1})}\le r\},
    $
    is pre-compact in $\mathcal{X}_S$. Setting
    $$
    B_{n}^1(r)=\{ u_{n}:\bigg\Vert\int_0^t \theta_R(\Vert u_{n}\Vert_{W^{1,\infty}})\sigma(t,u_{n})\mathrm{d}W(t) \bigg\Vert_{W^{\alpha,2p}([0,T];H^{s-1})}\le r\},
    $$
    and
    $$
    B_{n}^2(r)=\{ u_{n}:\bigg\Vert u_{n}-\int_0^t \theta_R(\Vert u_{n}\Vert_{W^{1,\infty}})\sigma(t,u_{n})\mathrm{d}W(t) \bigg\Vert_{W^{1,2p}([0,T];H^{s-1})}\le r\}.
    $$
    Notice that $B_{n}^1(r)\cap B_{n}^2(r)\subset B_{n}(r)$, by virtue of Proposition \ref{prop3.2}, the uniform momentum estimates and the Chebyshev inequality, we bound
    \begin{equation}
        \begin{aligned}
            \mu_{n}^S(B_{n}(r)^c)&\le \mathbb{P}\bigg( u_{n}:\bigg\Vert\int_0^t \theta_R(\Vert u_{n}\Vert_{W^{1,\infty}})\sigma(t,u_{n})\mathrm{d}W(t) \bigg\Vert_{W^{\alpha,2p}([0,T];H^{s-1})}> r\bigg)\\
           &\quad \  +\mathbb{P}\bigg( u_{n}:\bigg\Vert u_{n}-\int_0^t \theta_R(\Vert u_{n}\Vert_{W^{1,\infty}})\sigma(t,u_{n})\mathrm{d}W(t) \bigg\Vert_{W^{1,2p}([0,T];H^{s-1})}> r\bigg)\\
           &\le \frac{C}{r},
        \end{aligned}
    \end{equation}
    where $C$ is a universal constant independent of $r$ and $n$. In other words, for any $\gamma>1$, there exists $r>0$ large enough such that the compact set $B_{n}(r) \in \mathcal{X}_S$ satisfies $\mu_{n}^S(B_{n}(r)) > 1-\gamma$, for any $n\in\N$. We infer that $\mu_{n}^S$ is a tight sequence on $\mathcal{X}$. Moreover, $\mu_W$ is trivially weakly compact and so is
tight on $\mu_W$. As a consequence, the sequence $\{\mu_{n}\}_{n\in\N}$ is tight on $\mathcal{X}$ , which completes the proof.
\end{proof}
\subsection{Global existence to the cut-off system \eqref{cutoff}}\label{section3.5}
In this section, we prove the global existence of pathwise solutions for \eqref{eq:galerkin-system} by using the stochastic compactness method.
\begin{lemma}[Skorokhod Embedding Theorem \cite{prato1992}] \label{lem-skorokhod}
    Let $ X $ be a complete, separable metric space. For an arbitrary sequence $\{\mu_n\}\subset \mathcal{P}r({X})$ such that $\{\mu_n\}$ is tight on $(X; \mathcal{B}(X))$,
there exists a subsequence $\{\mu_{N,n}\}$ which converges weakly to a probability measure $\mu$,
and a probability space $(\Omega,\mathcal{F},\mathbb{P})$ with $X$ valued Borel measurable random variables
$x_n$ and $x$, such that $\mu_n$ is the distribution of $x_n$, $\mu$ is the distribution of $x$, and
$x_n \to x$, $\mathbb{P}$-a.s.
\end{lemma}
We shall construct global martingale solutions of \eqref{eq:galerkin-system} as follows.
\begin{proposition}\label{martingale}
    Let $s>\frac{d}{2}+3$, $p>2$ and $R\ge 1$ be fixed. Suppose that $\mu_0\in \mathcal{P}r(H^{s})$ with $\int_{H^s}\Vert u\Vert_{H^s}^p\mathrm{d}\mu_0(u)<\infty$, and the assumptions \textbf{(A1)} and \textbf{(A2)} hold. Then there exists a stochastic basis $\widetilde{\mathcal{S}}=(\widetilde{\Omega},\widetilde{\mathcal{F}},\widetilde{\mathbb{P}},\{\widetilde{\mathcal{F}}_{t}\}_{t \ge 0},\widetilde{W})$ and an $H^s$ valued, $\widetilde{\mathcal{F}}_{t}$-predictable process $\widetilde{u}\in L^2(\Omega;C([0,\infty);H^s))$ such that $\widetilde{\mathbb{P}}\circ \widetilde{u}(0)^{-1}=\mathbb{P}\circ u_0^{-1}$, and for every $t\ge 0$ 
    \begin{equation}\label{pp3.6-r}
    \begin{aligned}
        &\widetilde{u}+\int_0^t\theta_R(\Vert \widetilde{u}\Vert_{W^{1,\infty}})\big((1-2\widetilde{u})\nabla S\cdot \nabla \widetilde{u}
+(\widetilde{u}-\widetilde{u}^2)\Delta S\big)\mathrm{d}\tau\\
&\quad \quad\quad\quad\quad\quad\quad\quad \quad=\widetilde{u}(0)+\int_0^t\theta_R(\Vert \widetilde{u}\Vert_{W^{1,\infty}})\sigma(\tau,\widetilde{u})\mathrm{d}\widetilde{W},\quad \widetilde{\mathbb{P}}\textrm{-a.s.}
    \end{aligned}
    \end{equation}
\end{proposition}
\begin{proof}
   With  proposition \ref{prop3.3} in hand, we apply Lemma \ref{lem-skorokhod} to a weakly convergence subsequence of $\{\mu_{n}\}_{n\in\N}$. We obtain a new probability space $(\widetilde{\Omega},\widetilde{\mathcal{F}},\widetilde{\mathbb{P}})$ on which we have a sequence of random elements $\{(\widetilde{u}_{n},\widetilde{W}_{n})\}$ converging almost surely in $\mathcal{X}$ to an element $(\widetilde{u},\widetilde{W})$, that is
   \begin{equation}\label{pp3.6-4}
       \widetilde{u}_{n}\to\widetilde{u}\quad \ \text{in}\ C([0,T];H^{s-1}),\ \widetilde{\mathbb{P}}\text{-a.s.},
   \end{equation}
   and
   \begin{equation}\label{pp3.6-1}
       \widetilde{W}_{n}\to\widetilde{W}\quad \ \text{in}\ C([0,T];\mathfrak{U}),\ \widetilde{\mathbb{P}}\text{-a.s.}
   \end{equation}
   We can verify as in \cite{bensoussan1995,gawarecki2011stochastic} that $(\widetilde{u}_{n},\widetilde{W}_{n})$ satisfies the $n$th order Galerkin approximation \eqref{eq:galerkin-system} relative to the stochastic basis $(\widetilde{\Omega}_{n},\widetilde{\mathcal{F}}_{n},\widetilde{\mathbb{P}}_{n},\{\widetilde{\mathcal{F}}_{n,t}\}_{t \ge 0},\widetilde{W}_{n})$ with $\widetilde{\mathcal{F}}_{n,t}$ the completion of $\sigma$-algebra generated by $\{(\widetilde{u}_{n}(s),\widetilde{W}_{n}(s)):s\le t\}$
   and the following convergence holds
   \begin{equation}\label{pp3.6-2}
       G_{n}(\widetilde{u}_{n}(t))\to G(\widetilde{u}(t))\quad \ \text{in}\ C([0,T];H^{s-2}),\ \mathbb{P}\text{-a.s.},
   \end{equation}
   \begin{equation}\label{pp3.6-3}
       F_{n}(\widetilde{u}_{n}(t))\to F(\widetilde{u}(t))\quad \ \text{in}\ C([0,T];H^{s-1}),\ \mathbb{P}\text{-a.s.}
   \end{equation}
   Indeed, we observe from \eqref{pp3.6-4} that the uniform estimates in Proposition \ref{prop3.2} remain to be valid for $\{\widetilde{u}_{n}\}_{n\in\N}$ with $\mathbb{P}$ replaced by $\widetilde{P}_{n}$. It follows from Lemma \ref{flandoli1995} that $\{\widetilde{u}_{n}\}_{n\in\N}$ is compact in $C([0,T];H^{s-1})$,
and hence we conclude from property of $G_n$ and $F_n$ that,
    for fixed $l\ge 1$
   \begin{equation}
       G_{l}(\widetilde{u}_{l})\to G(\widetilde{u}_{l}),\quad \ F_{l}(\widetilde{u}_{l})\to F(\widetilde{u}_{l})\quad \ \text{in}\ C([0,T];H^{s-1}),\ \widetilde{\mathbb{P}}\text{-a.s.}
   \end{equation}
   The convergence \eqref{pp3.6-2} and \eqref{pp3.6-3} can be obtained by using standard diagonal argument.
   Using the uniform bound in Proposition \ref{prop3.2} for $\{\widetilde{u}_{n}\}_{n\in\N}$ and the almost surely convergences
\eqref{pp3.6-4}-\eqref{pp3.6-3}, we can show that $(\widetilde{u}, \widetilde{W})$ solves \eqref{pp3.6-r}. For the technical details of this passage to the limit we refer to, for example, \cite{SMEP2} where the analysis is carried out for the Euler-Poincar\'{e} equation.
Applying those arguments to this case introduces no additional difficulties, so we omit the
details here.
\end{proof}
We also note that Proposition \ref{martingale} directly leads to new conclusions regarding the existence of martingale solutions for the stochastic Keller-Segel equation \eqref{SHKS-trans}.
\begin{remark}[Existence of martingale solutions]
    By introducing the stopping time
    \begin{equation*}
        \tau:=\inf\{t\ge 0:\|\widetilde{u}\|_{W^{1,\infty}}\ge R\},
    \end{equation*}
    we can show that the pair $(\widetilde{u},\widetilde{S})$ obtained from Proposition \ref{martingale} constitutes a local martingale solution to \eqref{SHKS-trans}. Clearly, unless the initial condition satisfies $\|\widetilde{u}(0)\|_{W^{1,\infty}}<R$, i.e., unless $\mu_0(\{u_0\in H^s(\T^d);\|{u}(0)\|_{W^{1,\infty}}<R\})=1$, we have $\widetilde{\P}(\tau=0)>0$.
\end{remark}
Having now established Proposition \ref{martingale}, we would now expect pathwise solutions to exist once we establish that solutions are pathwise unique.
\begin{proposition}\label{pathwise-uniqueness}
    Let $s>\frac{d}{2}+3$, $p\ge 2$ and $R>1$. Assume that the conditions \textbf{(A1)} and \textbf{(A2)} hold, and $u_0\in L^2(\Omega;H^s(\mathbb{T}^d))$ is a $\mathcal{F}_0$-measurable initial data. Let $u$ and $v$ be two global martingale solutions of \eqref{cutoff} associated to the same stochastic basis $\mathcal{S}=(\Omega,\mathcal{F},\mathbb{P},\{\mathcal{F}_t\}_{t\ge 0},W)$. If $\mathbb{P}\{ u(0)=v(0)=u_0\}=1$, with then $u$ and $v$ are indistinguishable, that is
    \begin{equation}\label{pp3.7-fact}
        \mathbb{P}\{u(t)=v(t);\forall t \ge 0\}=1.
    \end{equation}
\end{proposition}
\begin{proof}
    From Proposition \ref{prop3.2}, we have for every $T>0$
    \begin{equation}\label{pp3.7-1}
        \mathbb{E}\sup_{t\in[0,T]}(\Vert u(t)\Vert_{H^s}^2+\Vert v(t)\Vert_{H^s}^2)\le M<\infty,
    \end{equation}
where $M$ is a constant depending only on $\mathbb{E}\|u_0\|_{H^s}^2$, $R$, $\beta$ and $T$. However, continuity in time is only guaranteed for the $H^{s-1}$ norms of $u$ and $v$, which lead us to define the collection of stopping times
    \begin{equation}
        \xi_K=\inf\{t\ge 0:\Vert u(t)\Vert_{H^{s-1}}^2+\Vert v(t)\Vert_{H^{s-1}}^2\ge K\},\quad \ \forall K>0.
    \end{equation}
    Due to \eqref{pp3.7-1}, we have $\xi_K\to \infty$ almost surely as $K\to \infty$.

    Taking $\varphi=u-v$, we have
    \begin{equation}\label{pp3.7-2}
        \begin{aligned}
            &\mathrm{d}\varphi+H_1(u,v)\mathrm{d}t+H_2(u,v)\mathrm{d}t =[\theta_R(\Vert u\Vert_{W^{1,\infty}})\sigma(t,u)-\theta_R(\Vert v\Vert_{W^{1,\infty}})\sigma(t,v)]\mathrm{d}W(t).
        \end{aligned}
    \end{equation}
    where
    \begin{equation*}
    \begin{aligned}
        H_1(u,v)&=\theta_R(\Vert u\Vert_{W^{1,\infty}})((1-2u)\nabla S(u)\cdot \nabla u) -\theta_R(\Vert v\Vert_{W^{1,\infty}})((1-2v)\nabla S(v)\cdot \nabla v),\\
            H_2(u,v)&=\theta_R(\Vert u\Vert_{W^{1,\infty}})(u-u^2)\Delta S(u)-\theta_R(\Vert v\Vert_{W^{1,\infty}})(v-v^2)\Delta S(v).
        \end{aligned}
    \end{equation*}
    Applying the operator $\Lambda^{s-2}$ to \eqref{pp3.7-2} and using the It\^o formula to $\Vert \Lambda^{s-2}\varphi\Vert_{L^2}^2$, we find
    \begin{equation}
    \begin{aligned}
        \mathrm{d}\Vert \Lambda^{s-2}\varphi\Vert_{L^2}^2 
        & =2\sum_{k\ge 1}(\Lambda^{s-2}(\theta_R(\Vert u\Vert_{W^{1,\infty}})\sigma(t,u)e_k-\theta_R(\Vert v\Vert_{W^{1,\infty}})\sigma(t,v)e_k),\Lambda^{s-2}\varphi)\mathrm{d}W_k\\
        &\quad \ -2(\Lambda^{s-2} H_1(u,v),\Lambda^{s-2}\varphi)_2\mathrm{d}t-2(\Lambda^{s-2} H_2(u,v),\Lambda^{s-2}\varphi)_2\mathrm{d}t\\
        &\quad \ + \Vert \Lambda^{s-2} \theta_R(\Vert u\Vert_{W^{1,\infty}})\sigma(t,u)-\theta_R(\Vert v\Vert_{W^{1,\infty}})\sigma(t,v)\Vert_{L^2}^2\mathrm{d}t\\
        &\ =\sum_{k\ge 1}I_{1_k} \mathrm{d}W_k+I_2 \mathrm{d}t +I_3 \mathrm{d}t+I_4 \mathrm{d}t.
    \end{aligned}
    \end{equation}
    We first integrate the above equation on $[0,\tau\land \xi_K]$, then we take a supremum over $\tau\in[0,t]$ and use the Burkholder-Davis-Gundy inequality to find
    \begin{equation}
        \mathbb{E}\sup_{\tau\in[0,t\land\xi_K]}\Vert \varphi\Vert_{H^{s-2}}^2\le C\mathbb{E}\sum_{k\ge 1}\sup_{\tau\in[0,t]}\bigg( \int_0^{\tau\land \xi_K}\vert I_{1_k}\vert^2\mathrm{d}t\bigg)^{\frac{1}{2}}+\sum_{i=2}^{4}\mathbb{E}\sup_{\tau \in [0,t]}\int_0^{\tau\land \xi_K}\vert I_i\vert  \mathrm{d}r.
    \end{equation}
    For the term involving $I_1$, by virtue of the locally Lipschitz condition \textbf{(A2)}, we have
    \begin{equation}\label{pp3.7-J1}
        \begin{aligned}
            &\mathbb{E}\sum_{k\ge1}\sup_{\tau\in[0,t]}\bigg( \int_0^{\tau\land \xi_K}\vert I_{1_k}\vert^2\mathrm{d}t\bigg)^{\frac{1}{2}}\\
            &\quad\ \le C \sum_{k\ge 1}\mathbb{E}\bigg(\int_0^{t\land \xi_K}(\Lambda^{s-2}(\theta_R(\Vert u\Vert_{W^{1,\infty}})\sigma(t,u)e_k-\theta_R(\Vert v\Vert_{W^{1,\infty}})\sigma(t,v)e_k),\Lambda^{s-2}\varphi)^2\mathrm{d}\tau\bigg)^{\frac{1}{2}}\\
            &\quad \ \le C\mathbb{E}\bigg(\int_0^{t\land \xi_K}\Vert\Lambda^{s-2}\varphi\Vert_{L^2}^2\Vert \theta_R(\Vert u\Vert_{W^{1,\infty}})\sigma(t,u)-\theta_R(\Vert v\Vert_{W^{1,\infty}})\sigma(t,v)\Vert_{\mathcal{L}_2(\mathfrak{U};H^{s-2})}^2\mathrm{d}\tau\bigg)^{\frac{1}{2}}\\
            &\quad \ \le \frac{1}{2}\mathbb{E}\sup_{\tau\in[0,t\land \xi_K]}\Vert \varphi\Vert_{H^{s-2}}^2+C\mathbb{E}\int_0^{t\land \xi_K}\Vert \varphi\Vert_{H^{s-2}}^2\mathrm{d}\tau.
        \end{aligned}
    \end{equation}
    For $I_2$, we can rewrite it as
       \begin{align*}
            I_2&=-2(\theta_R(\Vert u\Vert_{W^{1,\infty}})-\theta_R(\Vert v\Vert_{W^{1,\infty}}))( \varphi,(1-2u)\nabla S(u)\cdot \nabla u)_{H^{s-2}}\\
            &\quad \ -2\theta_R(\Vert v\Vert_{W^{1,\infty}})(\Lambda^{s-2}\varphi,\Lambda^{s-2}((1-2u)\nabla S(u))\cdot \nabla \varphi)_2\\
            &\quad \ -2\theta_R(\Vert v\Vert_{W^{1,\infty}})(\Lambda^{s-2}\varphi,\Lambda^{s-2}((1-2u)\nabla S(u)-(1-2v)\nabla S(v))\cdot \nabla u)_2\\
            &=I_{21}+I_{22}+I_{23}.
       \end{align*}
    By using the Mean Value Theorem for $\theta_R(\cdot)$, the continuous embedding $H^{s-2}\hookrightarrow W^{1,\infty}$, and the fact of $\Vert \theta_R' \Vert_{L^{\infty}}\le C_R$, we have
    \begin{equation*}
        \begin{aligned}
            I_{21}+I_{23}\le C_s \Vert \theta_R'\Vert_{L^{\infty}}\Vert \varphi\Vert_{H^{s-2}}^2\Vert  u\Vert_{L^{\infty}}^2 \Vert u\Vert_{H^{s-2}}\le C_{s,R} \Vert \varphi\Vert_{H^{s-2}}^2.
        \end{aligned}
    \end{equation*}
    Commutating the operator $\Lambda^{s-2}$ with $(1-2u)\nabla S(u)$, we get by integrating by parts that
    \begin{equation*}
       \begin{aligned}
            &\vert(\Lambda^{s-2} \varphi,\Lambda^{s-2} ((1-2u)\nabla S(u)\cdot\nabla \varphi))_2\vert\\
            &\quad \ \le  \vert (\Lambda^{s-2}\varphi,[\Lambda^{s-2},((1-2u)\nabla S(u))]\nabla  \varphi)_2\vert  +\frac{1}{2}\vert ((\Lambda^{s-2} \varphi)^2,\nabla ((1-2u)\nabla S(u)))_2\vert \\
            &\quad \ \le C_s(\Vert \varphi \Vert_{H^{s-2}}\Vert  u\Vert_{H^{s-2}}^2\Vert \nabla \varphi\Vert_{L^{\infty}}+\Vert \varphi\Vert_{H^{s-2}}\Vert \nabla \varphi\Vert_{H^{s-3}}\Vert \nabla u\Vert_{L^{\infty}}^2 +\Vert \varphi\Vert_{H^{s-2}}^2\Vert u\Vert_{W^{1,\infty}}^2)\\
            &\quad \ \le C_s \Vert u\Vert_{H^{s-2}}^2 \Vert \varphi\Vert_{H^{s-2}}^2.
       \end{aligned}
    \end{equation*}
    Combining the last three estimates, we have
    \begin{equation}\label{pp3.7_J2}
        \mathbb{E}\sup_{\tau \in [0,t]}\int_0^{\tau\land \xi_K}\vert I_2\vert  \mathrm{d}r\le C_{s,R}\int_0^t\mathbb{E}\sup_{\tau\in[0,r\land \xi_K]}\Vert \varphi(\tau)\Vert_{H^{s-2}}^2\mathrm{d}r.
    \end{equation}
    By applying Lemma \ref{lem-kato}, we deduce
    \begin{equation*}
        \begin{aligned}
            \Vert H_2(u,v)\Vert_{H^{s-2}}&\le \Vert (\theta_R(\Vert u\Vert_{W^{1,\infty}})-\theta_R(\Vert v\Vert_{W^{1,\infty}}))(u-u^2)\Delta S(u)\Vert_{H^{s-2}}\\
            &\quad \ +\Vert \theta_R(\Vert v\Vert_{W^{1,\infty}})((u-u^2)\Delta S(u)-(v-v^2)\Delta S(v))\Vert_{H^{s-2}}\\
            &\le C_s (\Vert \theta_R'\Vert_{L^{\infty}}\Vert \varphi\Vert_{W^{1,\infty}}\Vert (u-u^2)\Delta S(u)\Vert_{H^{s-2}}\\
            &\quad \ +\Vert (u-v-u^2+v^2)\Delta S(u)\Vert_{H^{s-2}} +\Vert (v-v^2)(\Delta S(u)-\Delta S(v))\Vert_{H^{s-2}}\\
            &\le C_{s,R}\Vert \varphi\Vert_{H^{s-2}}(\Vert u\Vert_{H^{s-2}}^3+\Vert u\Vert_{H^{s-2}}^2+\Vert v\Vert_{H^{s-2}}^2),
        \end{aligned}
    \end{equation*}
    which implies that
        \begin{align}\label{pp3.7_J3}
           \mathbb{E}\sup_{\tau \in [0,t]}\int_0^{\tau\land \xi_K}\vert I_3\vert  \mathrm{d}r&\le C_{s,R}\mathbb{E}\sup_{\tau\in[0,t]}\int_0^{\tau\land \xi_K}\Vert \varphi\Vert_{H^{s-2}}^2(\Vert u\Vert_{H^{s-2}}^3+\Vert u\Vert_{H^{s-2}}^2+\Vert v\Vert_{H^{s-2}}^2)\mathrm{d}r\\
            &\le C_{s,R,K}\int_0^t\mathbb{E}\sup_{\tau\in[0,r\land\xi_K]}\Vert \varphi(\tau)\Vert_{H^{s-2}}^2\mathrm{d}r.\nonumber
        \end{align}
    For the term involving $I_4$, by virtue of the assumptions \textbf{(A1)} and \textbf{(A2)}, we have
    \begin{equation}\label{pp3.7_J4}
        \begin{aligned}
            \mathbb{E}\sup_{\tau\in[0,t]}\int_0^{\tau\land \xi_K}\vert I_4\vert \mathrm{d}r&\le C_s\mathbb{E}\int_0^{t\land\xi_K}\Vert \theta_R'\Vert_{L^{\infty}}^2\Vert \varphi\Vert_{W^{1,\infty}}^2\beta (\Vert u\Vert_{L^{\infty}})(1+\Vert u\Vert_{H^{s-2}}^2)\mathrm{d}r\\
            &\quad \ +C_s\mathbb{E}\int_0^{t\land \xi_K}\gamma(\Vert u\Vert_{W^{1,\infty}}+\Vert v\Vert_{W^{1,\infty}})\Vert \varphi\Vert_{H^{s-2}}^2\mathrm{d}r\\
            &\le C_{s,R,K}\int_0^t\mathbb{E}\sup_{\tau\in[0,r\land\xi_K]}\Vert \varphi(\tau)\Vert_{H^{s-2}}^2\mathrm{d}r.
        \end{aligned}
    \end{equation}
    Combining the estimates \eqref{pp3.7-J1}-\eqref{pp3.7_J4}, we see that
    \begin{equation*}
        \mathbb{E}\sup_{\tau\in[0,t\land\xi_K]}\Vert \varphi(\tau)\Vert_{H^{s-2}}^2\le C_{s,R,K}\int_0^t\mathbb{ E}\sup_{\tau\in[0,r\land\xi_K]}\Vert \varphi\Vert_{H^{s-2}}^2\mathrm{d}r.
    \end{equation*}
    By applying the Gronwall inequality we obtain
    \begin{equation*}
        \mathbb{E}\sup_{\tau\in[0,T\land\xi_K]}\Vert \varphi(\tau)\Vert_{H^{s-2}}^2 =0,
    \end{equation*}
    which implies that $\mathbb{E}\sup_{\tau\in[0,T]}\Vert \varphi(\tau)\Vert_{H^{s-2}}^2=0$ by the Monotone Convergence Theorem and the fact that $\xi_K\to\infty$ as $K\to \infty$. Since $T$ is arbitrary, \eqref{pp3.7-fact} follows, and the proof of uniqueness is therefore complete.
\end{proof}
\begin{lemma}[Gy\"ongy-Krylov lemma \cite{gyongy}]\label{lem-gyongy}
     Let $X$ be a polish space equipped with the Borel $\sigma$-algebra. Let $\{Y_j\}_{j\ge 0}$ be a sequence of $X$ valued random variables and $\{\mu_{j,l}\}_{j,l\ge 0}$ be the joint laws of $\{ Y_j\}_{j\ge 0}$. Then $\{ Y_j\}_{j\ge 0}$ converges in probability if and only if for every subsequence of $\{\mu_{j_k,l_k}\}_{k\ge 0}$, there exists a further subsequence which weakly converges to some $\mu\in\mathcal{P}r(X\times X)$ satisfying
    \begin{equation*}
        \mu(\{(u,v)\in X\times X,u=v\})=1.
    \end{equation*}
\end{lemma}
With Propositions \ref{martingale} and \ref{pathwise-uniqueness} in hand, we can now apply Lemma \ref{lem-gyongy}
to prove the existence of global pathwise solution to \eqref{cutoff}.
\begin{proposition}\label{prop-3.9}
    Suppose that the condition \textbf{(A1)} and \textbf{(A2)} hold. Let $s>\frac{d}{2}+3$, $p>2$ and $u_0\in L^p(\Omega;H^s(\mathbb{T}^d))$ be a $\mathcal{F}_0$-measurable $H^s(\mathbb{T}^d)$-valued random variable. Then for $R>1$, the cut-off problem \eqref{cutoff} has a unique global pathwise solution in the sense of Definition \ref{def-pathwise}.
\end{proposition}
\begin{proof}
    The main target is to show that if $\mathcal{S}=(\Omega,\mathcal{F},\mathbb{P},\{\mathcal{F}_t\}_{t\ge 0},W)$ is given in advance, for the approximate solutions $u_{n}$, we can first take limit $n\to \infty$, which is relative to the original probability space $(\Omega,\mathcal{F},\mathbb{P})$, to build a pathwise solution under the given basis $\mathcal{S}$.
    Define a sequence of measures
    \begin{equation*}
        \nu_{i,j}(\cdot)=\mathbb{P}\{(u_{i},u_{j})\in \cdot\},\quad \ \pi_{i,j}(\cdot)=\mathbb{P}\{(u_{i},u_{j},W)\in \cdot\},
    \end{equation*}
    on $\mathcal{B}(\mathcal{X}_S\times\mathcal{X}_S)$ and $\mathcal{B}(\mathcal{X}_S\times\mathcal{X}_S\times\mathcal{X}_W)$, where $\mathcal{X}_S=C([0,T];H^{s-1}(\mathbb{T}^d))$ and $\mathcal{X}_W=C([0,T];\mathfrak{U})$. With only minor modifications to the arguements in Proposition \ref{prop3.3}, we see that the collection $\{\pi_{i,j}\}_{i,j\ge 1}$ is weakly compact and hence tight. By lemma \ref{lem-skorokhod}, we can extract a subsequence (still denoted by itself) such that $\pi_{i,j}\to\pi\in\mathcal{P}r(\mathcal{X}_S\times\mathcal{X}_S\times\mathcal{X}_W)$, and further choose a probability space $(\widetilde{\Omega},\widetilde{\mathcal{F}},\widetilde{\mathbb{P}})$ on which there exists a sequence of random variables $(\widetilde{u}_{i},\widetilde{u}_{j},\widetilde{W}_{i,j})$ converging almost surely in $\mathcal{X}_S\times\mathcal{X}_S\times\mathcal{X}_W$ to a random variable $(\widetilde{u},\widetilde{u}^*,\widetilde{W})$ and
    \begin{equation*}
        \widetilde{\mathbb{P}}\{(\widetilde{u}_{i},\widetilde{u}_{j},\widetilde{W}_{i,j})\in\cdot\}=\pi_{i,j}(\cdot).
    \end{equation*}
    Observe that in particular, $\nu_{i,j}$ converges weakly to a measure $\nu$ defined by
    \begin{equation*}
        \nu(\cdot)=\widetilde{\mathbb{P}}\{(\widetilde{u},\widetilde{u}^*)\in\cdot\}.
    \end{equation*}
    Let $\widetilde{\mathcal{S}}=(\widetilde{\Omega},\widetilde{\mathcal{F}},\{\widetilde{\mathcal{F}}_{t}\}_{t\ge 0},\widetilde{\mathbb{P}},\widetilde{W})$ with $\widetilde{\mathcal{F}}_{t}=\sigma\{\widetilde{u}(r),\widetilde{u}^*(r),\widetilde{W}(r)\}_{r\in[0,t]}$. Following the argument at the beginning of Section \ref{section3.5}, we infer that both $(\widetilde{\mathcal{S}},\widetilde{u},\infty)$ and $(\widetilde{\mathcal{S}},\widetilde{u}^*,\infty)$ are the martingale solutions to \eqref{SHKS-trans}. Since $\widetilde{u}(0)=\widetilde{u}^*(0)$, it follows from Proposition \ref{pathwise-uniqueness} that
    \begin{equation*}
        \nu\{(\widetilde{u},\widetilde{u}^*)\in\mathcal{X}_S\times\mathcal{X}_S,\widetilde{u}=\widetilde{u}^*\}=1.
    \end{equation*}
    Therefore, we can conclude from Lemma \ref{lem-gyongy} that, on the original probability space $(\Omega,\mathcal{F},\mathbb{P})$, $u_{n}\to u$ almost surely in the topology of $\mathcal{X}_S$. Having obtained this convergence and referring again to \eqref{L2p-estimate}, we may thus show that $u$ is a pathwise solution of \eqref{eq:galerkin-system}.
\end{proof}

\subsection{Local existence for SHKS \eqref{SHKS-trans} with smooth initial profile}
The construction of local maximal pathwise solutions to SHKS \eqref{SHKS-trans} can now be presented.
\begin{proposition}\label{prop3.12}
    Let $s>\frac{d}{2}+3$, and $u_0\in L^2(\Omega;H^s(\mathbb{T}^d)$ be a $\mathcal{F}_0$-measurable $H^s(\mathbb{T}^d)$-valued random element. Assume that the condition \textbf{(A1)}-\textbf{(A2)} hold. Then the Cauchy problem \eqref{SHKS-trans} has a unique local maximal pathwise solution $(u,\{\tau_n\}_{n\ge 1},\xi)$ in the sense of Definition \ref{def-pathwise}.
\end{proposition}
\begin{proof}
    Based on Proposition \ref{martingale}-\ref{pathwise-uniqueness} and Lemma \ref{lem-gyongy}, by applying the similar argument in Proposition \ref{prop-3.9}, we can prove that for any given $R>1$, the fixed stochastic basis $\mathcal{S}$, there exists a global pathwise solution $u$ to the cut-off problem \eqref{cutoff}. It remains to remove the truncation function $\theta_R$ in \eqref{cutoff} to construct the local pathwise solutions for \eqref{SHKS-trans}.

    \textsc{Step 1:} We first suppose that for some deterministic $M>0$, $\Vert u_0(\omega)\Vert_{H^s}<M$, $\mathbb{P}$-a.s., and let $D>0$ be the embedding constant such that $\Vert u\Vert_{W^{1,\infty}}\le D \Vert u\Vert_{H^s}$. Let $R>DM$ and define the stopping time
    \begin{equation*}
        \tau_M=\inf\{t\ge 0:\Vert u\Vert_{H^s}>M\}.
    \end{equation*}
    Then $\tau_M$ is strictly positive almost surely, and $\Vert u\Vert_{W^{1,\infty}}\le D\Vert u\Vert_{H^s}\le DM<R$ for any $t\in[0,\tau_M]$. Hence $\theta_R(\Vert u\Vert_{W^{1,\infty}})=1$. Then the pair $(u,\tau_M)$ is a unique pathwise solution to \eqref{SHKS-trans}.

    \textsc{Step 2: }To pass to the general case $\Vert u_0\Vert_{H^s}<\infty$ almost surely, we consider the decomposition
    \begin{equation*}
        u_0^k=u_0 \mathbbm{1}_{\{k\le \Vert u_0\Vert_{H^s}< k+1\}},\ k\in\mathbb{N}.
    \end{equation*}
    Then for each $k$, by choosing $R>D(k+1)$, we can construct a unique local pathwise solution $(u^k,\tau_k)$ to \eqref{SHKS-trans} with initial profile $u_0^k$. We define the pair $(u,\tau)$ by
    \begin{equation*}
        u=\sum_{k\in\mathbb{N}}u^k\mathbbm{1}_{\{k\le \Vert u_0\Vert_{H^s}< k+1\}},\quad\tau=\sum_{k\in\mathbb{N}}\tau_k\mathbbm{1}_{\{k\le\Vert u_0\Vert_{H^s}<k+1\}}.
    \end{equation*}
    Since $u^k\in C([0,\tau_k];H^s(\mathbb{T}^d))$ almost surely, we infer that $u\in C([0,\tau];H^s(\mathbb{T}^d))$ almost surely. Using the fact of $\mathbb{E}\Vert u_0\Vert_{H^s}^2<\infty$, we infer that $1=\sum_{k\in\mathbb{N}}\mathbbm{1}_{\{k\le \Vert u_0\Vert_{H^s}<k+1\}}$ almost surely. We apply Fatou's lemma to the uniform estimates in Proposition \ref{prop3.2} to find that
\begin{equation}\label{pp3.10-uniform}
    {\mathbb{E}}\sup_{t\in[0,T]}\Vert {u}\Vert_{H^s}^{2p}+{\mathbb{E}}\Vert {u}\Vert_{W^{\alpha,2p}([0,T];H^{s-1})}^{2p}\le C_{s,p,R,T,u_0,\alpha},\quad \ \forall T>0.
\end{equation}Therefore, by applying the uniform bounds \eqref{pp3.10-uniform}, we have
    \begin{equation*}
        \begin{aligned}
            \sup_{t\in[0,\tau]}\Vert u(t)\Vert_{H^s}^2&=\sum_{k\in\mathbb{N}}\mathbbm{1}_{\{k\le \Vert u_0\Vert_{H^s}<k+1\}}\sup_{t\in[0,\tau_k]}\Vert u^k(t)\Vert_{H^s}^2\\
            &\le \sum_{k\in\mathbb{N}}\mathbbm{1}_{\{k\le \Vert u_0\Vert_{H^s}<k+1\}}C_{s,R,T,u_0,\alpha}=C_{s,R,T,u_0,\alpha}<\infty.
        \end{aligned}
    \end{equation*}
By taking expectations, we find that $u(\cdot\land\tau)\in L^2(\Omega;C([0,\infty);H^s(\mathbb{T}^d)))$. Moreover, we have
\begin{align*}
        u(t\land\tau)=&\sum_{k\in\mathbb{N}}\mathbbm{1}_{\{k\le \Vert u_0\Vert_{H^s}<k+1\}}u^k(t\land\tau_k)\\
        =&\sum_{k\in\mathbb{N}}\mathbbm{1}_{\{k\le \Vert u_0\Vert_{H^s}<k+1\}}\bigg(u_0^k-\int_0^{t\land \tau_k}((1-2u^k)\nabla S(u^k)\cdot\nabla u^k+(u^k-(u^k)^2)\Delta S(u^k))\mathrm{d}r\\
        & + \int_0^{t\land \tau_k}\sigma(r,u^k)\mathrm{d}W(r)\bigg)\\
        =&\sum_{k\in\mathbb{N}}\mathbbm{1}_{\{k\le \Vert u_0\Vert_{H^s}<k+1\}}K(t\land\tau,u^k)=K(t\land \tau,u),
\end{align*}
where
\begin{equation*}
    K(t\land\tau,u)=u_0-\int_0^{t\land \tau}((1-2u)\nabla S(u)\cdot\nabla u+(u-u^2)\Delta S(u))\mathrm{d}r+\int_0^{t\land\tau}\sigma(r,u)\mathrm{d}W(r).
\end{equation*}
Therefore, the pair $(u,\tau)$ constructed above is a unique local pathwise solution to \eqref{SHKS-trans}. Finally, the passage from $(u,\tau)$ to a maximal solution in the sense of Definition \ref{def-pathwise} is standard and may be carried out as in \cite{glattholtz2014,jacod2006,mikulevicius2004}. We complete the proof.
\end{proof}
\section{ Local well-posedness in $H^s(\mathbb{T}^d)$}\label{section3}
For $s>\frac{d}{2}+1$, we now establish the local existence of solutions for any initial data $u_0\in H^s(\mathbb{T}^d)$, which is $\mathcal{F}_0$ measurable. For this propose we consider a sequence of equations
\begin{equation}\label{sequence-eq}
    \left\{
    \begin{aligned}
        &\mathrm{d}u^j(t)+((1-2u^j)\nabla S(u^j)\cdot\nabla u^j\\
        &\quad \ +(u^j-(u^j)^2)\Delta S(u^j))\mathrm{d}t=\sigma(t,u^j)\mathrm{d}W(t),&&x\in \mathbb{T}^d,\ t>0,\\
        &u^j(0,x)=u^j_0(x)=P_{\frac{1}{j}}u_0(x), &&x\in \mathbb{T}^d,
    \end{aligned}
    \right.
\end{equation}
where $P_{\frac{1}{j}}$ is the projection operator defined by \eqref{def-p_n}. It follows from Proposition \ref{prop3.12} that the functions $u^j$ in \eqref{sequence-eq} are well-defined. Employing an abstract Cauchy criterion, we prove that the sequence $u^j$ convergences to an limit $u$ up to a stopping time $\tau$, with convergence in the topology of $C([0,\tau];H^s(\mathbb{T}^d)).$

For a given $T>0$, we define $$\tau_{j,k,T}=\tau_{j,T}\land \tau_{k,T},\quad \forall j,k\ge 1,$$ where $$\tau_{j,T}=\inf\{t\ge 0:\Vert u^j(t)\Vert_{H^s}\ge 2 +\Vert u^j_0\Vert_{H^s}\}\land T.$$ We now recall the following Cauchy criteria.

\begin{lemma}[Abstract Cauchy Lemma \cite{holtz2009,mikulevicius2004}]\label{abstract-cauchy-lemma}
    Suppose that $(X,\|\cdot\|_X)$ and $(Y,\|\cdot\|_Y)$ are Banach space with continuous embedding $Y\subset X$. Define the space
    \begin{equation*}
        \mathcal{E}(T):=C([0,T];X)\cap L^2(0,T;Y),
    \end{equation*}
    with the norm
    \begin{equation*}
        \|v\|_{\mathcal{E}(T)}:=\bigg(\sup_{t\in[0,T]}\|v(t)\|_{X}^2+\int_0^T\|v(t)\|_Y^2\dd t\bigg)^{\frac{1}{2}}.
    \end{equation*}
    Let $\{v_n\}_{n\in\N}$ be a sequence of $Y$-valued stochastic process so that for every $T>0$, $v_n\in\mathcal{E}(T)$ a.s. For any $M>1$ and $T>0$, define a collection of stopping times
    \begin{equation*}
        \mathcal{T}_n^{M,T}:=\{\tau\le T:\|v_n\|_{\mathcal{E}(\tau)}\le M+\|v_n(0)\|_X\},
    \end{equation*}
    and let $\mathcal{T}_{n,m}^{M,T}:=\mathcal{T}_n^{M,T}\cap \mathcal{T}_m^{M,T}$.

    (i) Suppose that for some $M>1$ and $T>0$ 
    \begin{equation}\label{ac-1}
        \lim_{n\to \infty}\sup_{m\ge n}\sup_{\tau\in\mathcal{T}_{m,n}^{M,T}}\mathbb{E}(\|v_n-v_m\|_{\mathcal{E}(T)})=0,
    \end{equation}
    \begin{equation}\label{ac-2}
        \lim_{S\to 0}\sup_{n\in\N}\sup_{\tau\in\mathcal{T}_n^{M,T}}\P(\|v_n\|_{\mathcal{E}(\tau\land S)}^2>\|v_n(0)\|_{X}^2+(M-1)^2)=0.
    \end{equation}
    Then there exists a stopping time $\tau$ with $\P(0<\tau\le T)=1$ and a process $v$ with $v(\cdot\land T)\in\mathcal{E}(\tau)$ such that
    \begin{equation*}
        \|v_{n_l}-v\|_{\mathcal{E}(\tau)}\to 0,\quad\text{a.s.}\ \ \text{for some subsequence}\ n_l\uparrow\infty,
    \end{equation*}
    \begin{equation*}
        \|v\|_{\mathcal{E}(\tau)}\le M+\sup_{n\in\N}\|v_n(0)\|_X,\quad \text{a.s.}
    \end{equation*}

    (ii) If, in addition to \eqref{ac-1} and \eqref{ac-2}, we also have
    \begin{equation}
        \sup_{n\in\N}\mathbb{E}(\|v_n(0)\|_X^p)<\infty,\quad\text{for some}\ p\in[1,\infty),
    \end{equation}
    then
    \begin{equation}
        \mathbb{E}(\|v\|_{\mathcal{E}(\tau)}^p)\le C_p\left(M^p+\sup_n\mathbb{E}(\|v_n(0)\|_X^p)\right).
    \end{equation}
    Furthermore, there exists a sequence of measurable sets $\Omega_l\uparrow\Omega$ such that
    \begin{equation}
        \sup_{l\in\N}\mathbb{E}\left(\mathbbm{1}_{\Omega_l}\|v_{n_l}\|_{\mathcal{E}(\tau)}^p\right)<\infty,
    \end{equation}
    \begin{equation}
        \lim_{l\to\infty}\mathbb{E}\left(\mathbbm{1}_{\Omega_l}\|v_{n_l}-v\|_{\mathcal{E}(\tau)}^p\right)=0.
    \end{equation}
\end{lemma}

In   view of Lemma \ref{abstract-cauchy-lemma}, we shall establish the essential convergence needed for Theorem \ref{result1} in the general case by verifying \eqref{ac-1} and \eqref{ac-2}. Without loss of generality, we may first consider the case where $\Vert u_0\Vert_{H^s}\le M$ holds almost surely for some fixed constant $M$, and then extend to the general case $\Vert u_0\Vert_{H^s}<\infty$ almost surely via a decomposition method outlined in Proposition \ref{prop3.12}. Furthermore, the boundedness property of the projection operator $\Vert P_{n}u\Vert_{H^s}\le C_s\Vert u\Vert_{H^s}$, implies the uniform bound $\sup_{j\ge 1}\Vert u_0^j\Vert_{H^s}\le C_s\Vert u_0\Vert_{H^s}\le C_sM$, which plays a pivotal role in the following estimates.
\begin{proposition}\label{ac-1-proof}
    The sequence $\{u^j\}_{j\ge1}$ constructed in \eqref{sequence-eq} satisfies the convergence property stated in \eqref{ac-1}.
\end{proposition}
\begin{proof}
    We fix arbitrary $j,k\ge 1$ and denote $w^{j,k}=u^j-u^k$. It follows from \eqref{sequence-eq} that
    \begin{equation}
        \left\{
        \begin{aligned}
            &\mathrm{d}w^{j,k}(t)+((1-2u^j)\nabla S(u^j)\cdot\nabla u^j-(1-2u^k)\nabla S(u^k)\cdot\nabla u^k\\
            &\quad \ +(u^j-(u^j)^2)\Delta S(u^j)-(u^k-(u^k)^2)\Delta S(u^k))\mathrm{d}t=(\sigma(t,u^j)-\sigma(t,u^k))\mathrm{d}W(t),\\
            &w^{j,k}(0)=P_ju_0-P_ku_0.
        \end{aligned}
        \right.
    \end{equation}
    Applying the It\^o lemma to $(\Lambda^s w^{j,k},\Lambda^s w^{j,k})_2$, we have
    \begin{equation}
        \begin{aligned}\label{ac1proof-estimates}
            \mathrm{d}\Vert w^{j,k}(t)\Vert_{H^s}^2
            &=2\sum_{k\ge 1}(\Lambda^s w^{j,k},\Lambda^s(\sigma(t,u^j)e_k-\sigma(t,u^k)e_k))_2\mathrm{d}W_k\\
            &\quad\ -2(\Lambda^sw^{j,k},\Lambda^s((1-2u^j)\nabla S(u^j)\cdot\nabla u^j-(1-2u^k)\nabla S(u^k)\cdot\nabla u^k))_2\mathrm{d}t\\
            &\quad\ -2(\Lambda^s w^{j,k},\Lambda^s((u^j-(u^j)^2)\Delta S(u^j)-(u^k-(u^k)^2)\Delta S(u^k)))_2\mathrm{d}t\\
            &\quad \ +\Vert \Lambda^s(\sigma (u^j)-\sigma(t,u^k))\Vert_{\mathcal{L}_2(\mathfrak{U};L^2)}^2\mathrm{d}t\\
            &=\sum_{k\ge 1}I_{1_k,s}\mathrm{d}W_k+(I_{2,s}+I_{3,s}+I_{4,s})\mathrm{d}t.
        \end{aligned}
    \end{equation}
    To bound the terms involving $I_{1,s}$, we make use of the Burkholder-Davis-Gundy inequality under condition (\textbf{A2}), demonstrating that for any finite stopping time $\tau$, the following inequality is satisfied
        \begin{align}\label{ac1proof-I_{1,s}}
            &\mathbb{E}\sum_{k\ge 1}\sup_{r\in[0,\tau]}\bigg\vert\int_0^r I_{1_k,s}\mathrm{d}W_k\bigg\vert\nonumber\\
            &\quad\ \le C_s\mathbb{E}\sum_{k\ge 1}\bigg(\int_0^{\tau}\bigg(\int_{\mathbb{T}^d}\Lambda^s w^{j,k}\cdot\Lambda^s (\sigma(t,u^j)e_k-\sigma(t,u^k)e_k)\mathrm{d}x\bigg)^2\mathrm{d}r\bigg)^{\frac{1}{2}}\nonumber\\
            &\quad \ \le C_s\mathbb{E}\bigg(\int_0^{\tau}\Vert w^{j,k}\Vert_{H^s}^2\Vert \sigma(t,u^j)-\sigma(t,u^k)\Vert_{\mathcal{L}_2(\mathfrak{U};H^s)}^2\bigg)^{\frac{1}{2}}\\
            &\quad \ \le C_s\mathbb{E}\bigg(\sup_{t\in[0,\tau]}\Vert w^{j,k}\Vert_{H^s}^2\int_0^{\tau}\gamma(\Vert u^j\Vert_{W^{1,\infty}}+\Vert u^k\Vert_{W^{1,\infty}})\Vert w^{j,k}\Vert_{H^s}^2\mathrm{d}r\bigg)^{\frac{1}{2}}\nonumber\\
            &\quad\ \le \frac{1}{2}\mathbb{E}\sup_{t\in[0,\tau]}\Vert w^{j,k}\Vert_{H^s}^2+C_s\mathbb{E}\int_0^\tau\gamma(\Vert u^j\Vert_{W^{1,\infty}}+\Vert u^k\Vert_{W^{1,\infty}})\Vert w^{j,k}\Vert_{H^s}^2\mathrm{d}r.\nonumber
        \end{align}
    The following transformation can be performed on $I_{2,s}$
    \begin{align*}
        I_{2,s}&=-2(\Lambda^sw^{j,k},\Lambda^s((1-2u^j)\nabla S(u^j)\cdot\nabla u^j-(1-2u^k)\nabla S(u^k)\cdot\nabla u^j)_2\\
            &\quad\ +(1-2u^k)\nabla S(u^k)\cdot\nabla u^j)-(1-2u^k)\nabla S(u^k)\cdot\nabla u^k))_2\\
            &=-2(\Lambda^s w^{j,k},\Lambda^s(((1-2u^k)\nabla S(w^{j,k})-2w^{j,k}\nabla S(u^j))\cdot\nabla u^j))_2\\
            &\quad \ -2(\Lambda^s w^{j,k},\Lambda^s((1-2u^k)\nabla S(u^k)\cdot \nabla w^{j,k}))_2.
    \end{align*}
Applying Lemma \ref{lem-kato}, the H\"{o}lder inequality, and the embedding $H^{s-1}\hookrightarrow L^\infty$, we estimate the nonlinear terms involving $I_{2,s}$ as follows
\begin{align*}
    &\int_{\mathbb{T}^d}\vert \Lambda^s w^{j,k}\cdot\Lambda^s(((1-2u^k)\nabla S(w^{j,k})-2w^{j,k}\nabla S(u^j))\cdot\nabla u^j))\vert\mathrm{d}x\\
        &\quad \ \le C_s\Vert w^{j,k}\Vert_{H^s}(\Vert u^{k}\Vert_{H^s}\Vert w^{j,k}\Vert_{H^s}\Vert \nabla u^j\Vert_{L^\infty}+\Vert u^k\Vert_{L^\infty}\Vert w^{j,k}\Vert_{L^\infty}\Vert \nabla u^j\Vert_{H^s}\\
        &\quad\ \quad\ +\Vert w^{j,k}\Vert_{H^s}\Vert u^j\Vert_{L^\infty}^2+\Vert w^{j,k}\Vert_{L^\infty}\Vert u^j\Vert_{H^s}^2)\\
        &\quad\ \le C_s(\Vert w^{j,k}\Vert_{H^s}^2\Vert u^k\Vert_{H^s}\Vert u^j\Vert_{H^{s}}+\Vert w^{j,k}\Vert_{H^s}\Vert u^k\Vert_{H^{s-1}}\Vert w^{j,k}\Vert_{H^{s-1}}\Vert u^j\Vert_{H^{s+1}}\\
        &\quad\ \quad\ +\Vert w^{j,k}\Vert_{H^s}^2\Vert u^k\Vert_{H^{s-1}}^2+\Vert w^{j,k}\Vert_{H^s}\Vert w^{j,k}\Vert_{H^{s-1}}\Vert u^j\Vert_{H^s}^2)\\
        &\quad\ \le C_s\Vert w^{j,k}\Vert_{H^s}^2(\Vert u^j\Vert_{H^s}^2+\Vert u^k\Vert_{H^s}^2)+C_s\Vert w^{j,k}\Vert_{H^{s-1}}^2\Vert u^j\Vert_{H^{s+1}}^2.
\end{align*}
Using the commutator estimates and integrating by parts, we also have
\begin{align*}
    &\int_{\mathbb{T}^d}\vert\Lambda^s w^{j,k}\cdot\Lambda^s((1-2u^k)\nabla S(u^k)\cdot \nabla w^{j,k})\vert\mathrm{d}x\\
        &\quad\ \le\int_{\mathbb{T}^d}\vert \Lambda^s w^{j,k}\cdot[\Lambda^s,(1-2u^k)\nabla S(u^k)]\nabla w^{j,k}\vert\mathrm{d}x\\
        &\quad\ \quad\ +\frac{1}{2}\int_{\mathbb{T}^d}\vert(\Lambda^s w^{j,k})^2\cdot\nabla ((1-2u^k)\nabla S(u^k))\vert\mathrm{d}x\\
        &\quad\ \le C_s(\Vert w^{j,k}\Vert_{H^s}\Vert [\Lambda^s,(1-2u^k)\nabla S(u^k)]\nabla w^{j,k}\Vert_{L^2}+\Vert w^{j,k}\Vert_{H^s}^2\Vert u^k\Vert_{L^\infty}\Vert \nabla u^k\Vert_{L^\infty})\\
        &\quad\ \le C_s(\Vert w^{j,k}\Vert_{H^s}\Vert\Lambda^s (1-2u^k)\nabla S(u^k)\Vert_{L^2}\Vert w^{j,k}\Vert_{L^\infty}+\Vert w^{j,k}\Vert_{H^s}^2\Vert u^k\Vert_{L^\infty}\Vert \nabla u^k\Vert_{L^\infty})\\
        &\quad\ \le C_s\Vert w^{j,k}\Vert_{H^s}^2\Vert u^k\Vert_{H^s}^2.
\end{align*}
Therefore, based on the last two estimates and the definition of $\tau_{j,k,T}$, we can conclude that
\begin{equation}
    \begin{aligned}
        &\mathbb{E}\int_0^{t\land \tau_{j,k,T}}\vert I_{2,s}\vert\mathrm{d}r\\
        &\quad\ \le C_s\mathbb{E}\int_0^{t\land\tau_{j,k,T}}\Vert w^{j,k}\Vert_{H^s}^2(\Vert u^j\Vert_{H^s}^2+\Vert u^k\Vert_{H^s}^2)+\Vert w^{j,k}\Vert_{H^{s-1}}^2\Vert u^j\Vert_{H^{s+1}}^2\mathrm{d}r\\
        &\quad\ \le C_s\mathbb{E}\int_0^{t\land\tau_{j,k,T}}2(2+C_sM)^2\Vert w^{j,k}\Vert_{H^s}^2\mathrm{d}r+C_s\mathbb{E}\int_0^{t\land\tau_{j,k,T}}\Vert w^{j,k}\Vert_{H^{s-1}}^2\Vert u^j\Vert_{H^{s+1}}^2\mathrm{d}r.
    \end{aligned}
\end{equation}
To estimate the terms involving $I_{3,s}$, we apply Lemma \ref{lem-kato} to deduce that
\begin{equation}
    \begin{aligned}
        \mathbb{E}\int_0^{t\land\tau_{j,k,T}}\vert I_{3,s}\vert\mathrm{d}r&\le C_s\mathbb{E}\int_0^{t\land\tau_{j,k,T}}\Vert w^{j,k}\Vert_{H^s}^2(\Vert u^j\Vert_{H^s}^2+\Vert u^k\Vert_{H^s}^2)\\
        &\le C_s\mathbb{E}\int_0^{t\land\tau_{j,k,T}}2(2+C_sM)^2\Vert w^{j,k}\Vert_{H^s}^2\mathrm{d}r.
    \end{aligned}
\end{equation}
In view of the locally Lipschitz condition \textbf{(A2)}, together with the monotonic property of $\gamma(\cdot)$ and and the Sobolev embedding $H^s\hookrightarrow W^{1,\infty}$, one can  estimate the term  involving $I_{4,s}$ as follows
\begin{equation}\label{ac1proof-I_{4,s}}
    \begin{aligned}
        \mathbb{E}\int_0^{t\land\tau_{j,k,T}}\vert I_{4,s}\vert\mathrm{d}r&\le C_s\mathbb{E}\int_0^{t\land\tau_{j,k,T}}\gamma(\Vert u^j\Vert_{W^{1,\infty}}+\Vert u^k\Vert_{W^{1,\infty}})\Vert w^{j,k}\Vert_{H^{s-2}}^2\mathrm{d}r\\
        &\le C_s\mathbb{E}\int_0^{t\land\tau_{j,k,T}}\gamma(4+2C_sM)\Vert w^{j,k}\Vert_{H^s}^2\mathrm{d}r.
    \end{aligned}
\end{equation}
    By combining the estimates \eqref{ac1proof-I_{1,s}}-\eqref{ac1proof-I_{4,s}} and setting $\tau=t\land\tau_{j,k,T}$ in \eqref{ac1proof-I_{1,s}}, we obtain from \eqref{ac1proof-estimates} the following result
    \begin{equation}
        \begin{aligned}
            \mathbb{E}\sup_{r\in[0,t\land\tau_{j,k,T}]}\Vert w^{j,k}\Vert_{H^s}^2&\le 2\mathbb{ E}\|w^{j,k}(0)\|_{H^s}^2+C_{s,M}\int_0^t\bigg(\mathbb{E}\sup_{t'\in[0,r\land \tau_{j,k,T}]}\Vert w^{j,k}(t')\Vert\\
            &\quad\ +\mathbb{E}\sup_{t'\in[0,r\land\tau_{j,k,T}]}\Vert w^{j,k}(t')\Vert_{H^{s-1}}^2\|w^{j}(t')\|_{H^{s+1}}^2\bigg)\mathrm{d}r.
        \end{aligned}
    \end{equation}
Applying the Gronwall inequality, we have
\begin{equation}\label{ac1proof-w-bound}
    \mathbb{E}\sup_{t\in[0,\tau_{j,k,T}]}\| w^{j,k}(t)\|_{H^s}^2\le C_{s,M,\gamma,T}\bigg(\mathbb{E}\| w^{j,k}(0)\|_{H^s}^2+\mathbb{E}\sup_{t'\in[0,\tau_{j,k,T}]}(\|w^{j,k}(t')\|_{H^{s-1}}^2\|u^j(t')\|_{H^{s+1}}^2)\bigg).
\end{equation}
Recall that $P_{n}$ is uniformly bounded with respect to $n\in\N$ and satisfies \eqref{eq:spectral-error}, which implies that $\lim_{j\to \infty}\sup_{k\ge j}\Vert w^{j,k}(0)\Vert_{H^s}^2=0$. Therefore, in order to verify \eqref{ac-1}, it suffices to show that in \eqref{ac1proof-w-bound}
\begin{equation}\label{ac1proof-remain}
    \lim_{j\to \infty}\sup_{k\ge j}\mathbb{E}\sup_{t'\in[0,\tau_{j,k,T}]}(\| w^{j,k}(t')\|_{H^{s-1}}^2\| u^k(t')\|_{H^{s+1}}^2)=0.
\end{equation}
To show this, we begin by applying It\^o formula to $ \| u^j\|_{H^{s+1}}^2$ to find that
\begin{equation}\label{ac1proof-L}
    \begin{aligned}
        \mathrm{d}\| u^j\|_{H^{s+1}}^2&=2\sum_{k\ge 1}(\Lambda^{s+1} u^j,\Lambda^{s+1}\sigma(t,u^j)e_k)_2\mathrm{d}W_k(t)\\
        &\quad\ -2(\Lambda^{s+1} u^j,\Lambda^{s+1}((1-2u^j)\nabla S(u^j)\cdot \nabla u^j))_2\mathrm{d}t\\
        &\quad\ -2(\Lambda^{s+1}u^j,\Lambda^{s+1}((u^j-(u^j)^2))\Delta S(u^j))_2\mathrm{d}t\\
        &\quad\ +\Vert \Lambda^{s+1}\sigma(t,u^j)\|_{\mathcal{L}_2(\mathfrak{U};L^2)}^2\mathrm{d}t\\
        &=\sum_{k\ge 1}L_{1_k,s+1}\mathrm{d}W_k(t)+(L_{2,s+1}+L_{3,s+1}+L_{4,s+1})\mathrm{d}t.
    \end{aligned}
\end{equation}
As a result, by an application of the It\^o product rule for \eqref{ac1proof-estimates} and \eqref{ac1proof-L} (with $s$ replaced by $s-1$ throughout), we have
\begin{equation}
    \begin{aligned}
        &\mathrm{d}(\Vert w^{j,k}\|_{H^{s-1}}^2\| u^k\|_{H^{s+1}}^2) =\sum_{k\ge 1}(\| w^{j,k}\|_{H^{s-1}}^2 L_{1_k,s+1}+\| u^k\|_{H^{s+1}}^2I_{1_k,s-1})\mathrm{d}W_k(t)\\
        &\quad\ +\sum_{i=2}^4\| w^{j,k}\|_{H^{s-1}}^2 L_{i,s+1}\mathrm{d}t+\sum_{i=2}^4\| u^k\|_{H^{s+1}}^2 I_{i,s-1}\mathrm{d}t +\sum_{k,l\ge 1}I_{1_k,s-1}L_{1_l,s+1}\mathrm{d}t.
    \end{aligned}
\end{equation}
Thus we have
\begin{align}\label{K_sum}
    &\mathbb{E}\sup_{r\in[0,t\land\tau_{j,k,T}]}(\| u^{j,k}(r)\|_{H^{s-1}}^2\|u^k(r)\|_{H^{s+1}}^2)\nonumber\\
        &\quad\ \le \mathbb{E}(\| u^{j,k}(0)\|_{H^{s-1}}^2\| u^k(0)\|_{H^{s+1}}^2)+\mathbb{E}\sup_{r\in[0,t\land\tau_{j,k,T}]}\bigg\vert \sum_{k\ge 1}\int_0^t\| w^{j,k}\|_{H^{s-1}}^2 L_{1_k,s+1}\mathrm{d}W_k(r)\bigg\vert\nonumber\\
        &\quad\ \quad\ +\mathbb{E}\sup_{r\in[0,t\land\tau_{j,k,T}]}\bigg\vert\sum_{k\ge 1}\int_0^t\| u^k\|_{H^{s+1}}^2I_{1_k,s-1}\mathrm{d}W_k(r)\bigg\vert+\sum_{i=2}^4\mathbb{E}\int_0^{t\land\tau_{j,k,T}}\| w^{j,k}\|_{H^{s-1}}^2 \vert L_{i,s+1}\vert\mathrm{d}r\\
        &\quad\ \quad\ +\sum_{i=2}^4\mathbb{E}\int_0^{t\land\tau_{j,k,T}}\| u^k\|_{H^{s+1}}^2 \vert I_{i,s-1}\vert\mathrm{d}r+\mathbb{E}\sum_{k,l\ge 1}\int_0^{t\land\tau_{j,k,T}}\vert I_{1_k,s-1}L_{1_l,s+1}\vert \mathrm{d}r\nonumber\\
        &\quad\ =\mathbb{E}(\| u^{j,k}(0)\|_{H^{s-1}}^2\| u^k(0)\|_{H^{s+1}}^2)+\sum_{i=1}^5 K_i.\nonumber
\end{align}
For the stochastic term $K_1$, we establish the following estimate through an application of the Burkholder-Davis-Gundy inequality combined with assumption \textbf{(A1)}
\begin{equation}\label{K_1}
    \begin{aligned}
        K_1&\le C\mathbb{E}\bigg(\sum_{k\ge 1}\int_0^{t\land\tau_{j,k,T}}\| w^{j,k}\|_{H^{s-1}}^4 L_{1k,s+1}^2\mathrm{d}r\bigg)^{\frac{1}{2}}\\
        &\le C_s\mathbb{E}\bigg(\int_0^{t\land\tau_{j,k,T}}\|w^{j,k}\|_{H^{s-1}}^4\|u^j\|_{H^{s+1}}^2\|\sigma(r,u^j)\|_{\mathcal{L}_2(\mathfrak{U};H^{s+1})}^2\mathrm{d}r\bigg)^{\frac{1}{2}}\\
        &\le \frac{1}{4}\mathbb{E}\sup_{r\in[0,t\land\tau_{j,k,T}]}\|w^{j,k}\|_{H^{s-1}}^2\|u^j\|_{H^{s+1}}^2\\
        &\quad\ +C_s\mathbb{E}\int_0^{t\land\tau_{j,k,T}}\|w^{j,k}\|_{H^{s-1}}^2\beta(\|u^j\|_{H^{s}})(1+\|u^j\|_{H^{s+1}}^2)\mathrm{d}r\\
        &\le \frac{1}{4}\mathbb{E}\sup_{r\in[0,t\land\tau_{j,k,T}]}\|w^{j,k}\|_{H^{s-1}}^2\|u^j\|_{H^{s+1}}^2+C_{s,M,\beta,T}\mathbb{E}\sup_{r\in[0,t\land\tau_{j,k,T}]}\| w^{j,k}(r)\|_{H^{s-1}}^2\\
        &\quad\ +C_{s,M,\beta}\int_0^{t}\mathbb{E}\sup_{t'\in[0,t\land\tau_{j,k,T}]}\|w^{j,k}(t')\|_{H^{s-1}}^2\|u^j(t')\|_{H^{s+1}}^2\mathrm{d}r.
    \end{aligned}
\end{equation}
Following the similar argument as above, we derive the following bound for $K_2$
\begin{equation}
    \begin{aligned}
        K_2\le& \frac{1}{4}\mathbb{E}\sup_{r\in[0,t\land\tau_{j,k,T}]}\|w^{j,k}\|_{H^{s-1}}^2\|u^j\|_{H^{s+1}}^2\\
        & +C_{s,M,\gamma}\int_0^{t}\mathbb{E}\sup_{t'\in[0,r\land\tau_{j,k,T}]}\|w^{j,k}(t')\|_{H^{s-1}}^2\|u^j(t')\|_{H^{s+1}}^2\mathrm{d}r.
    \end{aligned}
\end{equation}
After using the Sobolev embedding $H^s\hookrightarrow W^{1,\infty}$ for $s>\frac{d}{2}+1$, combined with the estimates from Lemma \ref{lem-algebra} and Lemma \ref{lem-kato}, we obtain the following bounds for the integral terms $K_3$
\begin{equation}
    \begin{aligned}
        K_3&\le C_s \mathbb{E}\int_0^{t\land\tau_{j,k,T}}\| w^{j,k}\|_{H^{s-1}}^2\|u^j\|_{H^{s+1}}^2(\|u^j\|_{H^s}^2+\beta(\|u^j\|_{H^s}))\mathrm{d}r\\
        &\quad\ +C_s\mathbb{E}\int_0^{t\land\tau_{j,k,T}}\| w^{j,k}\|_{H^{s-1}}^2\beta(\|u^j\|_{H^s})\mathrm{d}r\\
        &\le C_{s,M,\beta}\int_0^t\mathbb{E}\sup_{t'\in[0,r\land\tau_{j,k,T}]}\| w^{j,k}(t')\|_{H^{s-1}}^2\|u^j(t')\|_{H^{s+1}}^2\mathrm{d}r\\
        &\quad\ +C_{s,M,\beta,T}\mathbb{E}\sup_{r\in [0,t\land\tau_{j,k,T}]}\| w^{j,k}(r)\|_{H^{s-1}}^2.
    \end{aligned}
\end{equation}
For the terms $K_4$, we can repeat the procedure as in \eqref{ac1proof-I_{1,s}}-\eqref{ac1proof-I_{4,s}} to find that
\begin{align}
    K_4&\le C_s\mathbb{E}\int_0^{t\land\tau_{j,k,T}}\|w^{j,k}\|_{H^s}^2(\|u^j\|_{H^s}^2+\|u^k\|_{H^s}^2)+\|w^{j,k}\|_{H^{s-1}}^2\|u^j\|_{H^{s+1}}^2\mathrm{d}r\nonumber\\
        &\quad\ +C_s\mathbb{E}\int_0^{t\land\tau_{j,k,T}}\gamma(\|u^j\|_{W^{1,\infty}}+\|u^k\|_{W^{1,\infty}})\| u^{j,k}\|_{H^{s-1}}^2\mathrm{d}r\nonumber\\
        &\le C_s\int_0^t\mathbb{E}\sup_{t'\in[0,r\land\tau_{j,k,T}]}\|w^{j,k}(t')\|_{H^{s-1}}^2\|u^j(t')\|_{H^{s+1}}^2\mathrm{d}r\\
        &\quad\ +C_{s,M,\gamma,T}\mathbb{E}\sup_{r\in[0,t\land\tau_{j,k,T}]}\|w^{j,k}(r)\|_{H^{s-1}}^2.\nonumber
\end{align}
Finally, building upon assumptions \textbf{(A1)}-\textbf{(A2)}, we establish the following estimate for the integral term $K_5$
\begin{equation}\label{K_5}
    \begin{aligned}
K_5&\le\mathbb{E}\int_0^{t\land\tau_{j,k,T}}\gamma(\|u^j\|_{H^s}+\|u^k\|_{H^s})\|w^{j,k}\|_{H^{s-1}}^2\beta(\|u^j\|_{H^s})(1+\|u^j\|_{H^{s+1}})\|u^j\|_{H^{s+1}}\mathrm{d}t\\
        &\le C_{s,\beta,\gamma}\mathbb{E}\int_0^{t\land\tau_{j,k,T}}(\|w^{j,k}\|_{H^{s-1}}^2\|u^j\|_{H^{s+1}}^2+\|w^{j,k}\|_{H^{s-1}}^2)\mathrm{d}t\\
        &\le C_{s,M,\beta,\gamma}\mathbb{E}\int_0^{t}\sup_{t'\in[0,r\land\tau_{j,k,T}]}(\|w^{j,k}(t')\|_{H^{s-1}}^2\|u^j(t')\|_{H^{s+1}}^2+\|w^{j,k}(t')\|_{H^{s-1}}^2)\mathrm{d}t\\
        &\le C_{s,M,\beta,\gamma,T}\bigg(\mathbb{E}\sup_{r\in[0,t\land\tau_{j,k,T}]}\|w^{j,k}(r)\|_{H^{s-1}}^2+\int_0^t\mathbb{E}\sup_{t'\in[0,r\land\tau_{j,k,T}]}\| w^{j,k}\|_{H^{s-1}}^2\|u^j\|_{H^{s+1}}^2\mathrm{d}r\bigg).
    \end{aligned}
\end{equation}
Combining the estimates \eqref{K_1}-\eqref{K_5} into \eqref{K_sum}, we find

    \begin{align}
        &\mathbb{E}\sup_{r\in[0,\tau_{j,k,T}]}\|w^{j,k}(r)\|_{H^{s-1}}^2\|u^j(r)\|_{H^{s+1}}^2\nonumber\\
        &\quad\ \le 2\mathbb{E}\|w^{j,k}(0)\|_{H^{s-1}}^2\|u^j(0)\|_{H^{s+1}}^2+C_{s,M,\beta,\gamma,T}\bigg(\mathbb{E}\sup_{r\in[0,t\land\tau_{j,k,T}]}\|w^{j,k}(r)\|_{H^{s-1}}^2\\
        &\quad\ +\int_0^t\mathbb{E}\sup_{t'\in[0,r\land\tau_{j,k,T}]}\| w^{j,k}\|_{H^{s-1}}^2\|u^j\|_{H^{s+1}}^2\mathrm{d}r\bigg).\nonumber
    \end{align}

By employing the Gronwall inequality, we obtain the following estimate
\begin{equation}\label{ac1gronwall}
    \begin{aligned}
        \mathbb{E}\sup_{r\in[0,\tau_{j,k,T}]}\|w^{j,k}(r)\|_{H^{s-1}}^2\|u^j(r)\|_{H^{s+1}}^2&\le C_{s,M,\beta,\gamma,T}\mathbb{E}\|w^{j,k}(0)\|_{H^{s-1}}^2\|u^j(0)\|_{H^{s+1}}^2\\
        &\quad\ +C_{s,M,\beta,\gamma,T}\mathbb{E}\sup_{r\in[0,\tau_{j,k,T}]}\|w^{j,k}(r)\|_{H^{s-1}}^2.
    \end{aligned}
\end{equation}
Notice that for any $u_0\in H^s(\mathbb{T}^d)$, the following fundamental properties are satisfied:
\begin{equation*}
    \| u_0-u_0^j\|_{H^{s-1}}^2=o(\frac{1}{j^2}),\quad\ \text{and}\ \|u^j(0)\|_{H^{s+1}}^2\le C_s j^2\|u_0\|_{H^s}^2,
\end{equation*}
from which we can bound the the first term on the right-hand side of \eqref{ac1gronwall} as follows
\begin{equation}\label{4.27}
    \begin{aligned}
        &\lim_{j\to\infty}\sup_{k\ge j}\mathbb{E}\|w^{j,k}(0)\|_{H^{s-1}}^2\|u^j(0)\|_{H^{s+1}}^2\\
        &\quad\ \le C_s\| u_0\|_{H^s}^2\lim_{j\to \infty}\sup_{k\ge j}\mathbb{E}(j^2\|u_0-u^j_0\|_{H^{s-1}}^2+\frac{j^2}{k^2}\cdot k^2\|u_0-u_0^j\|_{H^{s-1}}^2)=0.
    \end{aligned}
\end{equation}
For the second term, analogous reasoning to that employed in Proposition \ref{pathwise-uniqueness} demonstrates that it admits the bound $C\|u_0^j-u_0^k\|_{H^{s-1}}^2$, where the positive constant $C$ depends neither on $j$ nor on $k$. This leads to the following estimate
\begin{equation}
    \lim_{j\to\infty}\sup_{k\ge j}\mathbb{E}\sup_{r\in[0,\tau_{j,k,T}]}\|w^{j,k}(r)\|_{H^{s-1}}^2\le C\lim_{j\to\infty}\sup_{k\ge j}\mathbb{E}\|u^j_0-u^k_0\|_{H^{s-1}}^2=0,
\end{equation}
Combining the above result with the estimates \eqref{ac1gronwall} and \eqref{4.27}, we prove that \eqref{ac1proof-remain} holds. Therefore, by applying \eqref{ac1proof-w-bound} we can obtain the desired convergence conditions \eqref{ac-1}.
\end{proof}

\begin{proposition}\label{ac-2-proof}
    The sequence $\{u^j\}_{j\ge1}$ constructed in \eqref{sequence-eq} satisfies the convergence property stated in \eqref{ac-2}.
\end{proposition}
\begin{proof}
    Recall \eqref{ac1proof-L} (with $s+1$ replaced by $s$ throughout). Then for any $R>0$, we have
    \begin{equation}
        \sup_{t\in[0,\tau_{j,k,T}\land R]}\|u^j(t)\|_{H^s}^2\le \| u_0^j\|_{H^s}^2+\sup_{t\in[0,\tau_{j,k,T}\land R]}\bigg\vert \sum_{k\ge 1}\int_0^t L_{1_k,s}\mathrm{d}W_k(t)\bigg\vert+\sum_{i=2}^4\int_0^{\tau_{j,k,T}\land R}\vert L_{i,s}\vert\mathrm{d}t,
    \end{equation}
    which implies that
        \begin{align}
            &\mathbb{P}\bigg\{\sup_{t\in[0,\tau_{j,k,T}\land R]}\| u^k(t)\|_{H^s}^2>\|u_0^j\|_{H^s}^2+1\bigg\}\nonumber\\
            &\quad\ \le \mathbb{P}\bigg\{\sup_{t\in[0,\tau_{j,k,T}\land R]}\bigg\vert \sum_{k\ge 1}\int_0^t L_{1_k,s}\mathrm{d}W_k(t)\bigg\vert>\frac{1}{2}\bigg\}+\mathbb{P}\bigg\{\sum_{i=2}^4\int_0^{\tau_{j,k,T}\land R}\vert L_{i,s}\vert\mathrm{d}t>\frac{1}{2}\bigg\}.
        \end{align}
    Applying the Chebyshev inequality, Lemma \ref{lem-kato}, the condition \textbf{(A1)} and the fact of $H^s(\mathbb{T}^d) \hookrightarrow W^{1,\infty}(\mathbb{T}^d)$, we see that
    \begin{equation}
        \begin{aligned}
            \mathbb{P}\bigg\{\sum_{i=2}^4\int_0^{\tau_{j,k,T}\land S}\vert L_{i,s}\vert\mathrm{d}t>\frac{1}{2}\bigg\}&\le 2(T\land R)\mathbb{E}\sup_{t\in[0,\tau_{j,k,T\land R}]}\sum_{i=2}^4\vert L_{i,s}\vert\\
            &\le C_s(T\land R)\mathbb{E}\sup_{t\in[0,\tau_{j,k,T\land R}]}\Big((\| \nabla u\|_{L^\infty}^2+\|u\|_{L^{\infty}}^2)\|u\|_{H^s}^2\\
            &\quad\ +\beta(\| u^j\|_{W^{1,\infty}})(1+\|u^j\|_{H^s}^2)\Big)\\
            &\le C_{s,M,\beta}(T\land R).
        \end{aligned}
    \end{equation}
    To analyze the stochastic integral term, we employ Doob's maximal inequality together with the It\^o isometry property and the Minkowski inequality to derive the required estimate
    \begin{equation}
        \begin{aligned}
            \mathbb{P}\bigg\{\sup_{t\in[0,\tau_{j,k,T}\land R]}\bigg\vert \sum_{k\ge 1}\int_0^t L_{1_k,s}\mathrm{d}W_k(t)\bigg\vert>\frac{1}{2}\bigg\}&\le 4\mathbb{E}\bigg(\sum_{k\ge 1}\int_0^{\tau_{j,k,T}}\vert L_{1_k,s}\vert\mathrm{d}W_k(t)\bigg)^2\\
            &\le C_s\mathbb{E}\int_0^{\tau_{j,k,T}\land R}\| u^j\|_{H^s}^2\beta (\|u^j\|_{W^{1,\infty}})(1+\| u^j\|_{H^s}^2)\mathrm{d}t\\
            &\le C_s\mathbb{E}\int_0^{\tau_{j,k,T}\land R}C_{s,M}\mathrm{d}t\le C_{s,M,\beta}(T\land R).
        \end{aligned}
    \end{equation}
    Combining the above estimates and the fact that $T\land R\to0$ as $R\to0$, we conclude that condition \eqref{ac-2} is satisfied, which completes the proof of Proposition \ref{ac-2-proof}.
\end{proof}
\begin{proof}[\textbf{\emph{Proof of Theorem \ref{result1}.}}]
Building upon Propositions \ref{ac-1-proof}-\ref{ac-2-proof} and assuming $\|u_0\|_{H^s}\le M$ for some deterministic constant $M>0$. Lemma \ref{abstract-cauchy-lemma} guarantees the existence of a strictly positive stopping time $\tau$ with $\mathbb{P}(0<\tau\le T)=1$ for which the following convergence holds
\begin{equation}\label{pointwise}
    u^j\to u\quad\ \text{in}\ C([0,\tau];H^s(\mathbb{T}^d))\quad\ \text{as}\ j\to\infty,\quad\ \mathbb{P}\text{-a.s.},
\end{equation}
along with the uniform bound
\begin{equation}
    \sup_{j\ge 1}\sup_{t\in[0,\tau]}\|u^j(t)\|_{H^s}\le 2+C_sM<\infty,\quad\ \mathbb{P}\text{-a.s.}
\end{equation}
Since each approximate solution $u^j$ are continuous $\mathcal{F}_t$-adapted $H^s$-valued processes, it follows that the processes $u^j$ are $\mathcal{F}_t$-predictable. The preservation of measurability under pointwise limits then implies, through \eqref{pointwise}, that the limiting process $u$ is also $\mathcal{F}_t$-predictable.
By adapting the methodology developed in Section \ref{section3.5}, we establish that $(u,\tau)$ constitutes a local pathwise solution as specified in Definition \ref{def-pathwise}. The extension to maximal pathwise solutions follows from a standard argument (cf. \cite{glattholtz2014,jacod2006}). This completes the proof.
\end{proof}
\section{Global-in-time solutions}\label{section4}
\subsection{Global result for nonlinear multiplicative noise}
In this subsection, we establish the regularization effect induced by nonlinear multiplicative noise on the $t$-variable of solutions to the stochastic hyperbolic Keller-Segel equation.
\begin{proof}[\textbf{\emph{Proof of Theorem \ref{result2}.}}]
    Based on the conditions specified in Theorem \ref{result2}, it follows from Theorem \ref{result1} that the system \eqref{SHKS-trans} possesses a local strong pathwise solution $(u,\{\tau_n\}_{n\ge 1},\xi)$ in the sense of Definition \ref{def-pathwise}, with $\xi$ representing the maximum existence time. To complete the proof of Theorem \ref{result2}, it remains to demonstrate that $\xi=\infty$, $\mathbb{P}$-almost surely.

    To this end, we begin by applying the projection operator $P_n$ and then the differential operator $\Lambda^s$ to \eqref{SHKS-trans} to obtain
    \begin{equation}
        \begin{aligned}
            &\mathrm{d}\Lambda^sP_{n}u+\Lambda^s P_{n} ((1-2u)\nabla S(u)\cdot \nabla u)\mathrm{d}t\\
            &\quad\ +\Lambda^s P_{n} ((u-u^2)\Delta S(u))\mathrm{d}t=\sum_{i=1}^{\infty}c_i(1+\|u\|_{W^{1,\infty}})^{\delta}\Lambda^s P_{n}u\mathrm{d}W_i(t).
        \end{aligned}
    \end{equation}
    By using the It\^o formula, we arrive at
        \begin{align*}
            &\| P_{n}u(t)\|_{H^s}^2-\|P_{n}u_0\|_{H^s}^2\\
            &\quad\ =2\sum_{i=1}^\infty \int_0^tc_i(1+\|u\|_{W^{1,\infty}})^{\delta}\|P_{N }u\|_{H^s}^2\mathrm{d}W_i(r)+\sum_{i=1}^{\infty}\int_0^t\|\Lambda^s P_{n} c_i(1+\|u\|_{W^{1,\infty}})^{\delta}u\|_{L^2}^2\mathrm{d}r\\
            &\quad\ \quad\ -2\int_0^t(\Lambda^s P_{n} u,\Lambda^s(P_{n}((1-2u)\nabla S(u)\cdot \nabla u)+P_{n}((u-u^2)\Delta S(u)))_2\mathrm{d}r.
        \end{align*}
    To establish the global existence of the solution, we employ the It\^o formula once more to the logarithmic functional $\ln (e+\|P_{n}u(t)\|_{H^s}^2)$,  which yields
    \begin{equation}\label{logarithmic}
        \begin{aligned}
            &\ln(e+\| P_{n}u(t)\|_{H^s}^2)-\ln(e+\|P_{n}u_0\|_{H^s}^2)\\
            &\quad\ =-2\int_0^t\frac{(\Lambda^s P_{n}u,\Lambda^s(P_{n}((1-2u)\nabla S(u)\cdot \nabla u)+P_{n}((u-u^2)\Delta S(u)))_2}{e+\|P_{n}u(r)\|_{H^s}^2}\mathrm{d}r\\
            &\quad\ \quad\ +\sum_{i=1}^{\infty}\int_0^t\frac{\|\Lambda^s P_{n}c_i(1+\| u\|_{W^{1,\infty}})^{\delta}u\|_{L^2}^2}{e+\|P_{n}u(r)\|_{H^s}^2}\mathrm{d}r-2\sum_{i=1}^{\infty}\int_0^t\frac{c_i^2(1+\|u\|_{W^{1,\infty}})^{2\delta}\|P_{n}u\|_{H^s}^4}{(e+\|P_{n}u(r)\|_{H^s}^2)^2}\mathrm{d}r\\
            &\quad\ \quad\ +2\sum_{i=1}^{\infty}\int_0^t\frac{c_i(1+\|u\|_{W^{1,\infty}})^{\delta}\|P_{n}u\|_{H^s}^2}{e+\|P_{n}u(r)\|_{H^s}^2}\mathrm{d}W_i(r).
        \end{aligned}
    \end{equation}

    Furthermore, by employing commutator estimates along with the Sobolev embedding $H^s(\mathbb{T}^d)\hookrightarrow W^{1,\infty}(\mathbb{T}^d)$ for $s>\frac{d}{2}+1$, and following the steps as in Lemma 2.4 in \cite{SMEP2}, we derive that
    \begin{equation}\label{5.3}
        \begin{aligned}
            &|(P_{n}\Lambda^s u,P_{n}\Lambda^s((1-2u)\nabla S(u)\cdot \nabla u)+P_{n}((u-u^2)\Delta S(u)))_2|\le \kappa\|u\|_{W^{1,\infty}}\|u\|_{H^s}^2,
        \end{aligned}
    \end{equation}
    where $\kappa$ is a constant independent of $n$. By integrating \eqref{logarithmic} over the time interval $[0,t\land\tau_l]$, where the stopping time $\tau_l$ is defined as
    \begin{equation*}
        \tau_l=\inf\{t\ge 0:\| u(t)\|_{H^s}\ge l\},\quad\ \forall l\in\mathbb{N},
    \end{equation*}
    after taking the mathematical expectation on both sides, we deduce from the inequality \eqref{5.3} that
    \begin{equation}\label{5.4}
        \begin{aligned}
            &\mathbb{E}\ln(e+\| P_{n}u(t\land\tau_l)\|_{H^s}^2)-\mathbb{E}\ln(e+\|P_{n}u_0\|_{H^s}^2)\\
            &\quad\ \le\mathbb{E}\int_0^{t\land \tau_l}\frac{\kappa\|u\|_{W^{1,\infty}}\|u\|_{H^s}^2}{e+\|P_{n}u(r)\|_{H^s}^2}\mathrm{d}r+\mathbb{E}\sum_{i=1}^{\infty}c_i^2 \int_0^{t\land\tau_l}\frac{(1+\|u\|_{W^{1,\infty}})^{2\delta}\|P_{n}u\|_{H^s}^2}{e+\|P_{n}u(r)\|_{H^s}^2}\mathrm{d}r\\
            &\quad\ -\mathbb{E}\sum_{i=1}^\infty c_i^2\int_0^{t\land\tau_l}\frac{2(1+\|u\|_{W^{1,\infty}})^{2\delta}\|P_{n}u\|_{H^s}^4}{(e+\|P_{n}(r)\|_{H^s}^2)^2}\mathrm{d}r.
        \end{aligned}
    \end{equation}
    As the sequence of functions $\{P_{n}u\}_{n\in\N}$ convergence to $u$ in $C([0,T];H^s(\mathbb{T}^d))$ almost surely as $n\to\infty$. We may pass to the limit $n\rightarrow \infty$ in inequality \eqref{5.4} by applying the Dominated Convergence Theorem, and we get
    \begin{equation}\label{5.5}
        \begin{aligned}
            &\mathbb{E}\ln (e+\|u(t\land\tau_l)\|_{H^s}^2)-\mathbb{E}\ln(e+\|u_0\|_{H^s}^2)\\
            &\quad\ \le \mathbb{E}\int_0^{t\land\tau_l}\bigg(\vartheta(u(r))-\frac{\sum_{i=1}^{\infty}c_i^2(1+\|u\|_{W^{1,\infty}})^{2\delta}\|u\|_{H^s}^4}{(e+\|u\|_{H^s}^2)^2(1+\ln(e+\|u\|_{H^s}^2))}\bigg)\mathrm{d}r,
        \end{aligned}
    \end{equation}
     where $\vartheta(u)$ is defined as
    \begin{equation*}
        \begin{aligned}
            \vartheta(u)&\triangleq\frac{\kappa\| u\|_{W^{1,\infty}}\|u\|_{H^s}^2+\sum_{i=1}^{\infty}c_i^2(1+\|u\|_{W^{1,\infty}})^{2\delta}\|u\|_{H^s}^2}{e+\|u\|_{H^s}^2}-\frac{2\sum_{i=1}^{\infty}c_i^2(1+\|u\|_{W^{1,\infty}})^{2\delta}\| u\|_{H^s}^4}{(e+\|u\|_{H^s}^2)^2}\\
            &\quad\ +\frac{\sum_{i=1}^{\infty}c_i^2(1+\|u\|_{W^{1,\infty}})^{2\delta}\|u\|_{H^s}^4}{(e+\|u\|_{H^s}^2)^2(1+\ln(e+\|u\|_{H^s}^2))}.
        \end{aligned}
    \end{equation*}
    \emph{\textbf{Case 1}}: First we consider the case when $\delta>\frac{1}{2}$ and $\sum_{i=1}^\infty c_i^2\neq 0$. It is observed that the function $f(x)=\frac{x}{e+x}$ is is strictly monotonically increasing for all $x>0$. Combining the Sobolev embedding $\|u\|_{W^{1,\infty}}\le D\|u\|_{H^s}$, where $D$ is the Sobolev embedding constant as in the proof of Proposition \ref{prop3.12}, we deduce that for any $\delta>\frac{1}{2}$
    \begin{align*}
         \vartheta(u)&\le\kappa \|u\|_{W^{1,\infty}}+\sum_{i=1}^{\infty}c^2(1+\|u\|_{W^{1,\infty}})^{2\delta}+\sum_{i=1}^\infty c_i^2\frac{(1+\|u\|_{W^{1,\infty}})^{2\delta}}{1+\ln(e+\|u\|_{H^s}^2)}\\
            &\quad\ -2\sum_{i=1}^\infty c_i^2(1+\|u\|_{W^{1,\infty}})^{2\delta}\bigg(\frac{\|u\|_{H^s}^2}{e+\|u\|_{H^s}^2}\bigg)^2\\
            &\le(1+\|u\|_{W^{1,\infty}})^{2\delta}\bigg(\frac{\kappa\|u\|_{W^{1,\infty}}}{(1+\|u\|_{W^{1,\infty}})^{2\delta}}+\sum_{i=1}^\infty c_i^2-2\sum_{i=1}^\infty c_i^2\bigg(\frac{\|u\|_{H^s}^2}{e+\|u\|_{H^s}^2}\bigg)^2\\
            &\quad\ +\sum_{i=1}^\infty c_i^2\frac{1}{1+\ln(e+\|u\|_{H^s}^2)}\bigg)\\
            &\le (1+\|u\|_{W^{1,\infty}})^{2\delta}\bigg(\frac{\kappa(1+\|u\|_{W^{1,\infty}})}{(1+\|u\|_{W^{1,\infty}})^{2\delta}}+\sum_{i=1}^\infty c_i^2-2\sum_{i=1}^\infty c_i^2\bigg(\frac{\|u\|_{W^{1,\infty}}^2}{D^2e+\|u\|_{W^{1,\infty}}^2}\bigg)^2\\
            &\quad\ +\frac{\sum_{i=1}^\infty c_i^2}{1-2\ln(D)+\ln(D^2e+\|u\|_{W^{1,\infty}}^2)}\bigg).
    \end{align*}
    In order to analyze the right-hand side of the above formula, we define the following function
    \begin{equation*}
        h(x)=\frac{\kappa(1+x)}{(1+x)^{2\delta}}+\sum_{i=1}^\infty c_i^2-2\sum_{i=1}^\infty c_i^2\bigg(\frac{x^2}{D^2e+x^2}\bigg)^2+\frac{\sum_{i=1}^\infty c_i^2}{1-2\ln(D)+\ln(D^2e+x^2)}.
    \end{equation*}
    It is obvious that the function $h(x)$ is continuous when $x\ge 0$, and $\lim_{x\to\infty}h(x)=-\sum_{i=1}^\infty c_i^2$. Therefore, we can estimate $\vartheta(u)$ as follows
    \begin{equation*}
        \lim_{\|u\|_{W^{1,\infty}}\to\infty}\vartheta(u)\le \lim_{\|u\|_{W^{1,\infty}}\to\infty}(1+\|u\|_{W^{1,\infty}})^{2\delta}h(\|u\|_{W^{1,\infty}})=-\infty.
    \end{equation*}
    Based on the fact that $h(0)>0$, we obtain that there exists a constant $C$ such that $\vartheta(u)\le C$ holds.

    \emph{\textbf{Case 2}}: For the case where $\delta=\frac{1}{2}$ and $\sum_{i=1}^\infty c_i^2>\kappa$. the function $h(x)$ is expressed as
    \begin{equation*}
        h(x)=\kappa+\sum_{i=1}^\infty c_i^2-2\sum_{i=1}^\infty c_i^2\bigg(\frac{x^2}{D^2e+x^2}\bigg)^2+\frac{\sum_{i=1}^\infty c_i^2}{1-2\ln(D)+\ln(D^2e+x^2)}.
    \end{equation*}
    As $x$ approaches infinity, this function asymptotically behaves as $\kappa-\sum_{i=1}^\infty c_i^2$, with $\kappa$ being the positive universal constant defined in \eqref{5.3}, i.e., the limit of $h(x)$ becomes $\kappa-\sum_{i=1}^\infty c_i^2<0$ as $x\rightarrow\infty$, which indicates that $\vartheta(u)$ tends to negative infinity as $\|u\|_{W^{1,\infty}}$ approaches infinity. Consequently, the functional $\vartheta(u)$ is upper-bounded by a positive constant.

    In both cases, we conclude from \eqref{5.5} that there exists a positive constant $C$ such that
    \begin{equation}\label{5.6}
        \begin{aligned}
            &\mathbb{E}\ln (e+\|u(t\land\tau_l)\|_{H^s}^2)-\mathbb{E}\ln(e+\|u_0\|_{H^s}^2)\\
            &\quad\ \le C t- \mathbb{E}\int_0^{t\land\tau_l}\frac{\sum_{i=1}^\infty c_i^2(1+\|u\|_{W^{1,\infty}})^{2\delta}\|u\|_{H^s}^4}{(e+\|u\|_{H^s}^2)^2(1+\ln(e+\|u\|_{H^s}^2))}\mathrm{d}r.
        \end{aligned}
    \end{equation}
    As $P_{n}u$ converges to $u$ in $C([0,T];H^s(\mathbb{T}^d))$ when $n\to\infty$, we can find a non-negative function $\varpi$ satisfying $\lim_{n\to\infty}\varpi(n)=0$ such that the following inequality holds
        \begin{align}\label{5.7}
            &\bigg|\frac{\kappa\|u\|_{W^{1,\infty}}\|u\|_{H^s}^2+\sum_{i=1}^\infty c_i^2(1+\| u\|_{W^{1,\infty}})^{2\delta}\|P_{n}u\|_{H^s}^2}{e+\|P_{n}u\|_{H^s}^2}-\frac{2\sum_{i=1}^\infty c_i^2(1+\|u\|_{W^{1,\infty}})^{2\delta}\|P_{n}u\|_{H^s}^4)}{(e+\|P_{n}u\|_{H^s}^2)^2}\bigg|\nonumber\\
            &\quad\ =\bigg|\vartheta(u)+\varpi(n)-\frac{\sum_{i=1}^\infty c_i^2(1+\|u\|_{W^{1,\infty}})^{2\delta}\|u\|_{H^s}^4}{(e+\|u\|_{H^s}^2)^2(1+\ln(e+\|u\|_{H^s}^2))}\bigg|\\
            &\quad\ \le C+\varpi(n)+\frac{\sum_{i=1}^\infty c_i^2(1+\|u\|_{W^{1,\infty}})^{2\delta}\|u\|_{H^s}^4}{(e+\|u\|_{H^s}^2)^2(1+\ln(e+\|u\|_{H^s}^2))}.\nonumber
        \end{align}
    Applying the Burkholder-Davis-Gundy inequality to \eqref{logarithmic} and \eqref{5.7}, we obtain for any $T>0$,
    \begin{equation}
        \begin{aligned}
             \mathbb{E}\ln(e+\|P_{n}u(t)\|_{H^s}^2) 
            & \le\mathbb{E}\ln(e+\|P_{n}u_0\|_{H^s})+\frac{1}{2}\mathbb{E}(1+\ln(e+\|P_{n}u\|_{H^s}^2))\\
            & +\mathbb{E}\int_0^{t\land\tau_l}\bigg(C+\varpi(n)+\frac{\sum_{i=1}^\infty c_i^2(1+\|u\|_{W^{1,\infty}})^{2\delta}\|u\|_{H^s}^4}{(e+\|u\|_{H^s}^2)^2(1+\ln(e+\|u\|_{H^s}^2))}\bigg)\mathrm{d}r\\
            &  +C \mathbb{E}\int_0^{t\land\tau_l}\frac{\sum_{i=1}^\infty c_i^2(1+\|u\|_{W^{1,\infty}})^{2\delta}\|u\|_{H^s}^4}{(e+\|u\|_{H^s}^2)^2(1+\ln(e+\|u\|_{H^s}^2))}\mathrm{d}r\\
            &  \le C \mathbb{E}\ln(e+\|u_0\|_{H^s})+\frac{1}{2}\mathbb{E}(1+\ln(e+\|u\|_{H^s}^2))+(2C+\varpi(n))t\\
            &  +C \mathbb{E}\int_0^{t\land\tau_l}\frac{\sum_{i=1}^\infty c_i^2(1+\|u\|_{W^{1,\infty}})^{2\delta}\|u\|_{H^s}^4}{(e+\|u\|_{H^s}^2)^2(1+\ln(e+\|u\|_{H^s}^2))}\mathrm{d}r.
        \end{aligned}
    \end{equation}
    Taking a supremum over $t\in [0,T]$ and letting $n\to\infty$,  we establish the uniform estimate
    \begin{equation}
        \mathbb{E}\sup_{r\in[T\land\tau_l]}\ln(e+\|u(r)\|_{H^s}^2)\le C(\mathbb{E}\ln(e+\|u_0\|_{H^s}^2)+T).
    \end{equation}
    From the definition of the stopping time $\tau_l$, it follows that $\tau_l$ monotonically increases toward $\xi$ as $l\to\infty$. Applying Chebyshev's inequality and employing the uniform bound, we see that
        \begin{align}
            \mathbb{P}\{\xi<T\}&\le\mathbb{P}\{\tau_l\le T\}\nonumber\\
            &\le\mathbb{P}\bigg\{\sup_{t\in[0,T]}\|u(t)\|_{H^s}^2\ge l^2\bigg\}\nonumber\\
            &\le\mathbb{P}\bigg\{\sup_{t\in[0,T]}\ln(e+\|u(t)\|_{H^s}^2)\ge \ln(e+l^2)\bigg\}\\
            &\le\frac{\mathbb{E}\sup_{t\in[0,T]}\ln(e+\|u(t)\|_{H^s}^2)}{\ln(e+l^2)}\nonumber\\
            &\le C\frac{(\mathbb{E}\ln(e+\|u_0\|_{H^s}^2)+T)}{\ln(e+l^2)}\to0,\nonumber
        \end{align}
    as $l\to\infty$. Therefore, we conclude that $\mathbb{P}\{\xi<n\}=0$ holds for every   integer $n$, which leads to
    \begin{equation*}
        \mathbb{P}\{\xi=\infty\}=1-\mathbb{P}\bigg(\bigcup_{n\in\mathbb{N}}\{\xi<n\}\bigg)\ge 1-\sum_{n\in\mathbb{N}}\mathbb{P}\{\xi<n\}=1.
    \end{equation*}
    By following arguments similar to those in the proof of Theorem \ref{result1}, we establish that $\xi$ coincides with the maximal existence time in the sense of definition \ref{pathwise-uniqueness}. Consequently, the local pathwise solution $(u,\{\tau_n\}_{n\ge 1},\xi)$ extends globally in time. This completes the proof of Theorem \ref{result2}.
\end{proof}

\subsection{Global Existence for linear multiplicative noise}
Here we consider the stochastic hyperbolic Keller-Segel equation with linear multiplicative noise.
\begin{proof}[\textbf{\emph{Proof of Theorem \ref{result3}.}}]
To analyze \eqref{linear-noise}, we employ a stochastic transformation that converts the SPDE into a random PDE. Specifically, we introduce the auxiliary process
\begin{equation}
    \mu(t)=e^{\frac{\lambda^2}{2}t-\lambda W(t)}.
\end{equation}
Applying It\^o formula leads to
$
    \mathrm{d}\mu=\lambda^2 \mu\mathrm{d}t-\lambda \mu\mathrm{d}W(t).
$
Therefore, by using the It\^o product rule to $ \mu u $ we obtain
\begin{equation}
    \begin{aligned}
        \mathrm{d}(\mu u)&=\mu\mathrm{d}u+u\mathrm{d}\mu+\mathrm{d}\mu\mathrm{d}u\\
        &=-\mu((1-2u)\nabla S(u)\cdot\nabla u+(u-u^2)\Delta S(u))\mathrm{d}t\\
        &\quad\ +\lambda \mu u\mathrm{d}W(t)+u(\lambda^2 \mu\mathrm{d}t-\lambda \mu\mathrm{d}W(t))-\lambda^2\mu u \mathrm{d}t\\
        &=-\mu((1-2u)\nabla S(u)\cdot\nabla u+(u-u^2)\Delta S(u))\mathrm{d}t.
    \end{aligned}
\end{equation}
By defining $v=\mu u$, we transform equation \eqref{linear-noise} into the following random PDE:
\begin{equation}\label{random-pde}
    \left\{
    \begin{aligned}
        &\partial_t v+\mu^{-1}\nabla S(v)\cdot\nabla v-2\mu^{-2}v\nabla S(v)\cdot\nabla v\\
        &\quad\ +\mu^{-1}v\Delta S(v)-\mu^{-2}v^2\Delta S(v)=0,  &  x\in ~\mathbb{T}^d,t\ge 0,\\
    &v(0,x)=u_0(x),&  x\in ~\mathbb{T}^d,t=0.
    \end{aligned}
    \right.
\end{equation}
We begin by applying the Littlewood-Paley decomposition operator $\Delta_j$ (cf. \cite[Chapter 2]{Danchin2011}) to the first equation in \eqref{random-pde}, yielding
    \begin{align*}
        &\partial_t\Delta_j v+\mu^{-1}\nabla S(v)\cdot \nabla \Delta_j v-2\mu^{-2}v\nabla S(v)\cdot \nabla \Delta_j v\\
        &\quad\ =\mu^{-1}(\nabla S(v)\cdot\nabla \Delta_j v-\Delta_j(\nabla S(v)\cdot\nabla v))\\
        &\quad\ \quad\ -2\mu^{-2}(v\nabla S(v)\cdot\nabla\Delta_j v-\Delta_j(v\nabla S(v)\cdot\nabla v))\\
        &\quad\ \quad\ -\mu^{-1}\Delta_j(v\Delta S(v))+\mu^{-2}\Delta_j(v^2\Delta S(v)),\quad\ \P\text{-a.s.}
    \end{align*}
Multiplying both sides by $\Delta_j$ and   integrating over $\mathbb{T}^d$, we obtain
    \begin{align}\label{5.15}
        \frac{1}{2}\frac{\mathrm{d}}{\mathrm{d} t}\|\Delta_j v\|_{L^2}^2&=\frac{1}{2}\mu^{-1}(\div (\nabla S(v)),|\Delta_j v|^2)_2-\mu^{-2}(\div(v\nabla S(v)),|\Delta_j v|^2)_2\nonumber\\
        &\quad\ +\mu^{-1}(\nabla S(v)\cdot\nabla \Delta_j v-\Delta_j(\nabla S(v)\cdot\nabla v),\Delta_j v)_2\nonumber\\
        &\quad\ -2\mu^{-2}(v\nabla S(v)\cdot\nabla\Delta_j v-\Delta_j(v\nabla S(v)\cdot\nabla v),\Delta_jv)_2\nonumber\\
        &\quad\ -\mu^{-1}\Delta_j(v\Delta S(v),\Delta_jv)_2+\mu^{-2}\Delta_j(v^2\Delta S(v),\Delta_j v)_2\nonumber\\
        &\le \frac{1}{2}\mu^{-1}\|\div(\nabla S(v))\|_{L^{\infty}}\|\Delta v\|_{L^2}^2-\mu^{-2}\|\div(v\nabla S(v))\|_{L^{\infty}}\|\Delta_j v\|_{L^2}^2\\
        &\quad\ +C\mu^{-1}c_j2^{-js}\|\Delta_jv\|_{L^2}\|v\|_{B_{2,2}^s}\|\nabla v\|_{L^{\infty}}\nonumber\\
        &\quad\ -2C\mu^{-2}c_j2^{-js}\|\Delta_jv\|_{L^2}\|v\|_{B_{2,2}^s}\|\nabla v\|_{L^\infty}^2\nonumber\\
        &\quad\ -\mu^{-1}\|\Delta_j(v\Delta S(v))\|_{L^2}\|\Delta_jv\|_{L^2}\nonumber\\
        &\quad\ +\mu^{-2}\|\Delta_j(v^2\Delta S(v))\|_{L^2}\|\Delta_j v\|_{L^2},\quad\ \P\text{-a.s.},\nonumber
    \end{align}
where $c_j$ satisfies $\|\{c_j\}_{j\ge -1}\|_{l^2}=1$, and we have employed standard commutator estimates from \cite{Danchin2011} for the third and fourth term on the right-hand side of \eqref{5.15}, we derive the following differential inequality
\begin{equation}\label{5.16}
    \begin{aligned}
        \frac{\mathrm{d}}{\mathrm{d} t}2^{js-1}\|\Delta_j v\|_{L^2}&\le 2^{js-1}\mu^{-1}\|\nabla v\|_{L^{\infty}}\|\Delta v\|_{L^2}+2^{js}\mu^{-2}\|\nabla v\|_{L^{\infty}}^2\|\Delta_j v\|_{L^2}\\
        &\quad\ +C\mu^{-1}c_j\|v\|_{B_{2,2}^s}\|\nabla v\|_{L^{\infty}}+2C\mu^{-2}c_j\|v\|_{B_{2,2}^s}\|\nabla v\|_{L^\infty}^2\\
        &\quad\ +2^{js}\mu^{-1}\|\Delta_j(v\Delta S(v))\|_{L^2}+2^{js}\mu^{-2}\|\Delta_j(v^2\Delta S(v))\|_{L^2},\quad\ \P\text{-a.s.}
    \end{aligned}
\end{equation}
Taking the $l^2$-norm on both sides of \eqref{5.16} with respect to $j$ leads to
\begin{equation}\label{5.17}
\begin{aligned}
    \frac{\mathrm{d}}{\mathrm{d} t}\|v\|_{B_{2,2}^s}&\le C\mu^{-1}\|v\|_{B_{2,2}^s}\|\nabla v\|_{L^{\infty}}+2C\mu^{-2}\|v\|_{B_{2,2}^s}\|\nabla v\|_{L^\infty}^2\\
    &\quad\ +\mu^{-1}\|v\Delta S(v)\|_{B_{2,2}^s}+\mu^{-2}\|v^2\Delta S(v)\|_{B_{2,2}^s},\quad\ \P\text{-a.s.}
\end{aligned}
\end{equation}
Since $S(v)=(1-\Delta)^{-1}v$, we follow the similar method as in \cite[Lemma 2.4]{SMEP2} to obtain
\begin{equation}\label{5.18}
    \mu^{-1}\|v\Delta S(v)\|_{B_{2,2}^s}+\mu^{-2}\|v^2\Delta S(v)\|_{B_{2,2}^s}\le C(\|v\|_{L^{\infty}}+\|\nabla v\|_{L^{\infty}})\|v\|_{B_{2,2}^s}.
\end{equation}
Due to \eqref{5.17}, \eqref{5.18} and the fact that $B_{2,2}^s(\T^d)\approx H^s(\T^d)$, we deduce that
\begin{equation}\label{5.19}
    \begin{aligned}
        \frac{\mathrm{d}}{\mathrm{d} t}\|v\|_{H^s}&\le C\mu^{-1}\|v\|_{H^s}\| v\|_{W^{1,\infty}}+C^2\mu^{-2}\|v\|_{H^s}\|v\|_{W^{1,\infty}}^2, \quad\ \P\text{-a.s.},
    \end{aligned}
\end{equation}
where $C>0$ is a constant depending only on the regularity parameter $s$ and dimension $d$.

By defining $\varphi(t)\triangleq e^{-\frac{\lambda^2}{2}t}v(t)$, we transform \eqref{5.19} into the following form
\begin{equation}\label{5.20}
    \begin{aligned}
        \frac{\mathrm{d}}{\mathrm{d} t}\| \varphi\|_{H^s}+\frac{\lambda^2}{2}\|\varphi\|_{H^s}&\le Ce^{\lambda W(t)}\|\varphi\|_{H^s}\| \varphi\|_{W^{1,\infty}}+C^2e^{2\lambda W(t)}\|v\|_{H^s}\|v\|_{W^{1,\infty}}^2, \quad\ \P\text{-a.s.}
    \end{aligned}
\end{equation}
For any $\rho\ge 2$, we define the stopping time
\begin{equation}
    \begin{aligned}
        \tau^*&\triangleq\inf_{t\ge 0}\bigg\{ e^{\lambda W(t)}\| \varphi(t)\|_{W^{1,\infty}}\ge \frac{\lambda^2}{2\rho C}\bigg\}\\
        &=\inf_{t\ge 0}\bigg\{\|v(t)\|_{W^{1,\infty}}\ge \frac{\lambda^2}{2\rho C}\bigg\}.
    \end{aligned}
\end{equation}
Assume that $\|u_0\|_{H^s}\le \frac{\lambda^2}{4R\rho C D}$, where $D$ is the positive Sobolev embedding constant from $H^s$ to $W^{1,\infty}$, then we have
\begin{equation*}
    \|v_0\|_{W^{1,\infty}}\le D\|v_0\|_{H^s}\le \frac{\lambda^2}{4 R\rho C}.
\end{equation*}
From the definition of $\tau^*$, we observe that $\tau^* >0$ holds $\P$-almost surely. Moreover, for all $t\in[0,\tau^*)$, \eqref{5.20} indicates
\begin{equation*}
    \frac{\mathrm{d}}{\mathrm{d} t}\|\varphi(t)\|_{H^s}+(\frac{\lambda^2}{2}-\frac{\lambda^2}{2\rho}-\frac{\lambda^4}{4\rho^2})\|\varphi(t)\|_{H^s}\le 0,
\end{equation*}
then we obtain the following exponential bound
\begin{equation}\label{5.21}
    \begin{aligned}
        \|v\|_{H^s}&\le e^{\lambda W(t)-(\frac{\lambda^2}{2}-\frac{\lambda^2}{2\rho}-\frac{\lambda^4}{4\rho^2})t}\|v_0\|_{H^s}\\
        &\le e^{\lambda W(t)-\frac{1}{2}(\frac{\lambda^2}{2}-\frac{\lambda^2}{2\rho})t}e^{-\frac{1}{2}(\frac{\lambda^2}{2}-\frac{\lambda^2}{2\rho}-\frac{\lambda^4}{2\rho^2})t}\|v_0\|_{H^s}.
    \end{aligned}
\end{equation}
Define the stopping time
\begin{equation*}
    \widetilde{\tau}\triangleq\inf_{t\ge 0}\left\{e^{\lambda W(t)-\frac{1}{2}(\frac{\lambda^2}{2}-\frac{\lambda^2}{2\rho})t}\ge R\right\},
\end{equation*}
It is obvious that $\widetilde{\tau}>0$ $\P$-almost surely. Combining \eqref{5.21} with the condition on initial data yields
\begin{equation}
    \|v(t)\|_{H^s}\le \frac{\lambda^2R}{4R\rho CM}e^{-\frac{1}{2}(\frac{\lambda^2}{2}-\frac{\lambda^2}{2\rho}-\frac{\lambda^4}{2\rho^2})t}\le \frac{\lambda^2}{4\rho CM},\quad\ \forall t\in[0,\tau^*\land\widetilde{\tau}).
\end{equation}
The definition of $\tau^*$ implies that $\tau^*\land\widetilde{\tau}=\widetilde{\tau}$, and consequently
\begin{equation}
    \sup_{t\in[0,\widetilde{\tau}]}\|v(t)\|_{W^{1,\infty}}\le M\sup_{t\in[0,\widetilde{\tau}]} \|v(t)\|_{H^s}\le\frac{\lambda^2}{4\rho M},
\end{equation}
which implies that $\tau^*\ge \widetilde{\tau}$. Moreover, the maximal pathwise solution  $(u,\{\tau_n\}_{n\ge 1},\xi)$ is global in time on the set $\{\widetilde{\tau}=\infty\}$, which corresponds to the set where the process $\Phi_t=e^{\lambda W(t)-\frac{1}{2}(\frac{\lambda^2}{2}-\frac{\lambda^2}{2\rho})t}$ remains bounded by $R$ for all $t\ge 0$. To establish the existence of global solutions, we need to estimate $\P\{\widetilde{\tau}=\infty\}$.

On the set $\{\widetilde{\tau}=\infty\}$, we observe that $0<\Phi_t\le R$ holds for all $t\ge 0$. The process $\Phi_t$ obeys the geometric Brownian motion described by the following stochastic differential equation
\begin{equation}
    \mathrm{d}\Phi_t=\bigg(\frac{\lambda^2}{4}+\frac{\lambda^2}{4\rho}\bigg)\Phi_t\mathrm{d} t+\lambda\Phi_t\mathrm{d} W(t).
\end{equation}
Applying It\^o formula to $\Phi^k_t$ for $k\in\R$ yields
\begin{equation}\label{5.26}
    \begin{aligned}
        \mathrm{d} \Phi_t^k&=k\Phi^{k-1}(t)\mathrm{d} \Phi_t+\frac{k(k-1)}{2}\Phi^{k-2}_t\lambda^2\Phi_t^2\mathrm{d} t \\
        &=k\bigg(\frac{\lambda^2}{4}+\frac{\lambda^2}{4\rho}\bigg)\Phi^k_t\mathrm{d} t+\frac{\lambda^2k(k-1)}{2}\Phi^k_t\mathrm{d} t+\lambda k\Phi^k \mathrm{d} W(t).
    \end{aligned}
\end{equation}
Integrating \eqref{5.26} on $[0,t\land\widetilde{\tau}]$ and taking expectations, we obtain
\begin{equation}\label{5.27}
    \mathbb{E}(\Phi^k(t\land\widetilde{\tau}))=1+\mathbb{E}\int_0^{t\land\widetilde{\tau}}k\bigg(\frac{\lambda^2}{4}+\frac{\lambda^2}{4\rho}\bigg)\Phi^k_r\mathrm{d} r+\mathbb{E}\int_0^{t\land\widetilde{\tau}}\frac{\lambda^2k(k-1)}{2}\Phi^k_r\mathrm{d} r.
\end{equation}
Letting $k=\frac{1}{2}-\frac{1}{2\rho}$, the expectation \eqref{5.27} simplifies to $\mathbb{E}(\Phi^{\frac{1}{2}-\frac{1}{2\rho}}(t\land\widetilde{\tau}))=1$. Therefore we have
\begin{equation}
    \begin{aligned}
        \P\{\xi=\infty\}&\ge \P\{\widetilde{\tau}=\infty\}=\P\bigg\{\bigcap_n\{\widetilde{\tau}>n\}\bigg\}=\lim_{n\to\infty}\P\{\widetilde{\tau}>n\}\\
        &\ge \lim_{n\to\infty}\P\{\Phi_t^{\frac{1}{2}-\frac{1}{2\rho}}(n\land \widetilde{\tau})<R^{\frac{1}{2}-\frac{1}{2\rho}}\}\\
        &\ge 1-\lim_{n\to\infty}\P \{\Phi^{\frac{1}{2}-\frac{1}{2\rho}}_t(n\land\widetilde{\tau})\ge R^{\frac{1}{2}-\frac{1}{2\rho}}\}\\
        &\ge1-\frac{1}{R^{\frac{1}{2}-\frac{1}{2\rho}}}\lim_{n\to\infty}\mathbb{E}(\Phi^{\frac{1}{2}-\frac{1}{2\rho}}(t\land\widetilde{\tau}))=1-\frac{1}{R^{\frac{1}{2}-\frac{1}{2\rho}}}.
    \end{aligned}
\end{equation}
This completes the proof of Theorem \ref{result3}, demonstrating that the probability of global existence $\mathbb{P}\{\xi=\infty\}$ approaches $1$ as $R\to\infty$.
\end{proof}

\section{Acknowledgements}

The authors would like to express their gratitude to Professor Tusheng Zhang for his insightful  discussions on the initial draft of this manuscript. This work was partially supported by the National Key Research and Development Program of China (Grant No. 2023YFC2206100).

\bibliographystyle{plain}%
\bibliography{SHKS}

\end{document}